\documentclass[11pt]{amsart}
\usepackage{amsmath, amssymb}
\usepackage{amsfonts}
\usepackage{mathrsfs}

\usepackage{amsthm}
\usepackage{graphicx}

\usepackage[cmtip,arrow]{xy} 
\usepackage{pb-diagram,pb-xy}

\usepackage{color}

\usepackage{xcolor}

\allowdisplaybreaks
\newtheorem*{rep@theorem}{\rep@title}
\newcommand{\newreptheorem}[2]{%
\newenvironment{rep#1}[1]{%
 \def\rep@title{#2 \ref{##1}}%
 \begin{rep@theorem}}%
 {\end{rep@theorem}}}
\makeatother

\newreptheorem{theorem}{Theorem}
\newreptheorem{lemma}{Lemma}
\newreptheorem{proposition}{Proposition}
\newreptheorem{corollary}{Corollary}

\newtheorem{theorem}{Theorem}[section]
\newtheorem{lemma}[theorem]{Lemma}
\newtheorem{proposition}[theorem]{Proposition}
\newtheorem{corollary}[theorem]{Corollary}

\newtheorem*{theorem*}{Theorem}
\newtheorem*{proposition*}{Proposition}
\theoremstyle{remark}

\newtheorem{definition}[theorem]{Definition}

\newcommand{\cplx}{\mathbb{C}}
\newcommand{\conc}{\mathcal{C}}
\newcommand{\AC}{\mathcal{AC}}
\newcommand{\Z}{\mathbb{Z}}
\newcommand{\Q}{\mathbb{Q}}

\newcommand{\C}{\mathbb{C}}
\newcommand{\T}{\mathbb{T}}

\newcommand{\im}{\operatorname{im}}

\newcommand{\lnk}{\operatorname{lk}}
\newcommand{\eref}[1]{(\ref{#1})}

\newcommand{\iterate}[3]{\ensuremath{\underset{#1}{\overset{#2}{#3}}}}
\newcommand{\bdry}{\ensuremath{\partial}}

\newcommand{\onto}{\twoheadrightarrow}
\newcommand{\into}{\hookrightarrow}

\newcommand{\inter}{\operatorname{int}}

\renewcommand{\th}{$^{\text{th}}$ }

\definecolor{myPurple}{rgb}{139, 0, 204}

\numberwithin{equation}{section}

\begin{document}
\title[computing $\rho$-invariants of links]{Computing the rho-invariants of links via the signature of colored links with applications to the linear independence of twist knots.}

\author{Christopher William Davis}

\address{Department of Mathematics \\ Rice UNIVERSITY}
\email{youremailaddress@math.uh.edu}

\date{\today}


\begin{abstract}
We use a link invariant defined by Cimasoni-Florens to compute $\rho$-invariants.  This generalizes results of Cochran-Teichner and Friedl on  knots to the setting of links.  As an application, we prove with only twelve possible exceptions that the twist knots of algebraic order two are linearly independent in the topological concordance group.
\end{abstract}

\maketitle

\section{Introduction}

Von Neumann $\rho$-invariants associated to metabelian representations have played a central role in many recent constructions of knots which are algebraically slice and yet are not slice.  For example, see \cite{derivatives}, \cite{MyFirstPaper}, \cite{Fr3}.  These invariants, however, are difficult to compute.  A standard approach is to build a link (called a derivative in \cite{derivatives}) associated to that knot.  Metabelian invariants of the knot are then approximated by abelian invariants of this derivative. (See for example \cite[Corollary 5.9]{derivatives} of \cite[Theorem 5.6]{MyFirstPaper}).  In the case that the knot is genus 1, its derivative is also a knot and so existing tools which compute  abelian $\rho$-invariants of knots may be applied.  See \cite[Proposition 3.3]{Fr2}, \cite[Corollary 4.3]{Fr3}, \cite[Proposition 5.1]{structureInConcordance}.  

The goal of this manuscript is to extend these computational tools to the setting of links.  We provide results (Theorems~\ref{unitary case} and \ref{L2 Rho}) which compute abelian $\rho$-invariants of links in terms of a signature function defined by Cimasoni and Florens in \cite{CimFlo}.  We use these results to compute the metabelian invariant of \cite{MyFirstPaper} and find that with only twelve possible exceptions the twist knots (Figure~\ref{fig:twist}) which are of order 2 in the algebraic concordance group are linearly independent in the topological knot concordance group.

\subsection{Background and statement of results}

An $n$-component link is an isotopy class of smooth, ordered, oriented embeddings of a disjoint union of $n$ copies of the circle $ S^1 \sqcup \dots \sqcup  S^1$ into $S^3$.  A knot can be thought of as a 1-component link.  Two $n$-component links $L$ and $J$ are called \textbf{concordant} if there is a locally flat, ordered, oriented embedding of a disjoint union of $n$ annuli $\underset{n}{\sqcup} (S^1\times [0,1])$ into $S^3\times [0,1]$ sending $\underset{n}{\sqcup}(S^1\times\{0\})$ to $L$ in $S^3\times \{0\}$ and $\underset{n}{\sqcup}(S^1\times\{1\})$ to $J$ in $S^3\times \{1\}$.  If one restricts to knots then the set of concordance classes forms a group, $\conc$, called the knot concordance group, under the operation of connected sum.

  For a knot $K$, the \textbf{preferred longitude} of $K$ is the unique isotopy class of pushoffs of $K$ which are nullhomologous in the complement of $K$.  For a link $L$, the \textbf{zero surgery} on $L$, denoted $M(L)$, is defined by cutting out regular neighborhoods of the components of $L$ and gluing back in solid tori so that the preferred longitudes of the components of  $L$ bound disks. 
  
  For any closed oriented 3-manifold $M$, the von Neumann $\rho$-invariant assigns to each representation of the fundamental group of $M$ a number.  There are two varieties of $\rho$-invariant that interest us in this paper.  The {unitary $\rho$-invariant} takes a representation $\alpha$ from $\pi_1(M)$ to $U(n)$, the group of $n\times n$ unitary matrices and returns an integer, $\rho(M,\alpha)$.  The $L^2$-$\rho$-invariant takes a homomorphism from $\pi_1(M)$ to a discrete group $\Gamma$ and  returns a real number $\rho^{(2)}(M,\phi)$.  We recall some needed properties of $\rho$-invariants in subsection~\ref{rho invt facts}.

The invariants studied in this paper are $\rho$-invariants of $M(L)$ corresponding to either one-dimensional representations ($\alpha:\pi_1(M(L))\to U(1)$) or corresponding to homomorphisms to finitely generated abelian groups ($\phi:\pi_1(M(L))\to A$).  We make our computation in terms of a signature function of Cimasoni and Florens, which we now recall.


In \cite{Cooper1982, CooperThesis} Cooper gives a generalization of the concept of a Seifert surface (A surface bounded by a knot) to the setting of links.  Using this generalization, Cimasoni and Florens \cite{CimFlo} define a signature function associated to a link.  Much like the Tristram-Levine signature, it can be explicitly computed.  In Theorem~\ref{unitary case} we show that unitary $\rho$-invariants of zero surgery on a link with zero pairwise linking are given by evaluating this signature function.   In  Theorem~\ref{L2 Rho} we express $L^2$-$\rho$-invariants corresponding to homomorphisms to abelian groups in terms of integrals of this signature function.  

We need some notation before we can state our results.  By $\T^n$ we mean the  set of $n$-tuples of complex numbers of norm $1$ in $\cplx^n$.  By $\T^n_*$ we mean the set of all $n$-tuples of complex numbers of norm $1$ other than $1$ itself:
$$
\begin{array}{c}
\T^n = \{(\omega_1,\dots, \omega_n)\in \cplx^n\text{ such that for all }i, |\omega_i|=1\}\\
\T^n_* = \{(\omega_1,\dots, \omega_n)\in \cplx^n\text{ such that for all }i, |\omega_i|=1\text{ and } \omega_i\neq 1\}.
\end{array}
$$
In \cite{CimFlo}, Cimasoni and Florens associate to an $n$-colored link (a link whose every component has been colored with an integer between $1$ and $n$) a function $\sigma_L:\T^n_*\to \Z$, which we call  \textbf{Cimasoni-Florens signature function} of $L$.  Any $n$-component link has an $n$-coloring which colors the $i$\th component of $L$ with the integer $i$.  Unless otherwise stated links in this paper are assumed to be colored in this way.

Cimasoni and Florens  prove  that many one-dimensional unitary $\rho$-invariants of three-manifolds gotten by arbitrary surgery on links
 are given in terms of these invariants.  The specialization of their theorem to the case of zero surgery on links with zero pairwise linking gives:

\begin{proposition*}[see Theorem 6.7 of \cite{CimFlo}]
Let $L$ be an $n$-component link with zero pairwise linking numbers and $M(L)$ be the 3-manifold gotten by zero surgery on $L$.  For an integer $q>0$ and integers $p_1,\dots, p_n$ coprime to $q$ let $\alpha:H_1(M(L))\to U(1)$ be the representation sending the meridian of the $j$\th component of $L$ to $\omega_j=e^{2\pi i p_j/q}$.  Then
$$
\rho(M(L),\alpha) = \sigma_L(\omega)
$$
where $\omega = (\omega_1,\dots, \omega_n)$.
\end{proposition*}

Notice that this theorem does not apply for all cloices of $\omega$.  Take for example $\omega = (e^{2\pi i 4/5},e^{2\pi i 5/4})$.  Our first main result is Theorem~\ref{unitary case}, which allows us to compute more one dimensional $\rho$-invariants of zero surgery on links with zero pairwise linking.  

By $\T^n_\Q$ we  mean the set of all $n$-tuples of complex roots of unity:
$$
\T^n_\Q = \left\{(\omega_1,\dots, \omega_n)\in \cplx^n\text{ such that for all }j, \omega_j^k=1\text{ for some }k\right\}.
$$

\begin{reptheorem}{unitary case}
For any link $L$ with zero pairwise linking there is a function $\widehat\sigma_L:\T^n\to \Z$ such that
\begin{enumerate}
\item \label{gives rho} For $\omega\in \T^n_\Q$ and the representation $\alpha_{\omega}:H_1(M(L))\to U(1)$ sending $\mu_i\mapsto \omega_i$, $\rho(M(L),\alpha_{\omega}) = \widehat\sigma_L(\omega)$.
\item \label{agree} $\sigma_L$ and $\widehat \sigma_L$ agree on $\T^n_\Q\cap \T^n_*$.
\end{enumerate}
\end{reptheorem}

We give a precise meaning to $\widehat\sigma_L$ in Definition~\ref{define extension}.

Before we can state our second main result, Theorem~\ref{L2 Rho}, which expresses $L^2$-$\rho$-invariants in terms of $\widehat{\sigma_L}$, we need some additional notation.  Let $A$ be the abelian group presented as
$A=\langle g_1,\dots,g_n | r_1,\dots, r_l\rangle$
where $r_j=\displaystyle\sum_{i=1}^n {a_{i,j}} \cdot g_i$ for integers $a_{i,j}$ We define $\T_A$ be the set of all $n$-tuples in $\T^n$ satisfying the relations $r_1,\dots, r_l$: 
$$
\T_A=\{(\omega_1,\dots,\omega_n)\in \T^n\text{ such that for }1\le j\le l\text{ }, \prod_{i=1}^n \omega_i^{a_{i,j}} = 1\}.
$$
The space $\T_A$ is isomorphic (as a topological group) to the set of all representations from $A$ to $U(1)$.  The isomorphism is given by sending $(\omega_1,\dots, \omega_n)\in \T_A$ to the representation mapping the generator $g_i$ to  $\omega_i$.  In particular, as a topological group $\T_A$ is independent of presentation and consists of $k$ embedded disjoint copies of $\T^R$ in $\T^n$ if $A$ has rank $R$ and has $k$ elements in its torsion subgroup.  

The $R$-dimensional Haussdorff measure on $\T_A\subseteq \T^n$ is finite and translation invariant on $\T_A$.  By the uniqueness of Haar measure, see for example \cite[Theorem 11.4]{Conway1990}, the isomorphism from $\T_A$ to the group of unitary representations of $A$ sends normalized Haussdorff measure to normalized Haar measure.  This implies that normalized $R$-dimensional Hausdorff measure on $\T_A$ is also independent of presentation.

We prove the following theorem:

\begin{reptheorem}{L2 Rho}
Let $A=\langle g_1\dots g_n|r_1\dots r_l\rangle$ be a rank-$R$ abelian group.  Let $L$ be an $n$-component link with zero pairwise linking numbers.  Let $\phi:H_1(M(L))\to A$ be given by sending $\mu_i$ to $g_i$, then 
\begin{equation*}\rho^{(2)}(M(L),\phi)=\frac{1}{\lambda(\T_A)}\int_{\T_A}\widehat\sigma_L(\omega)d\lambda(\omega)\end{equation*}
where $\lambda$ is $R$-dimensional Hausdorff measure on $\T_A$.
\end{reptheorem}

  We outline an application of Theorem~\ref{L2 Rho} to the concordance of the twist knots of Figure~\ref{fig:twist}.    We begin by recalling the theory of algebraic concordance.  For a knot $K$ and any genus $g$ surface $F$ bounded by $K$, a nonseparating $g$-component link $L= L_1\cup\dots\cup L_g$ on $F$ is called a \textbf{derivative} of $K$ in \cite{derivatives} if the linking number between  $L_i$ and the result of pushing $L_j$ off of $F$ vanishes for all $i$ and $j$.  If $K$ has a derivative then $K$ is called \textbf{algebraically slice}.  In \cite{Le10} Levine shows that the quotient of $\conc$ (the knot concordance group) by the algebraically slice knots is isomorphic to $\Z^\infty\oplus \Z_2^\infty\oplus \Z_4^\infty$.  This quotient is called the \textbf{algebraic concordance group} and is denoted $\AC$.  In some sense all of this complexity is represented by the twist knots.  They generate a subgroup of $\AC$ isomorphic to $\Z^\infty\oplus \Z_2^\infty\oplus \Z_4^\infty$.  To be precise:
  \begin{itemize}
  \item $T_n$ is of infinite order in $\AC$ if $n<0$.
  \item $T_n$ is algebraically slice if $4n+1$ is a square.   
  \item $T_n$ is  of order 4  in $\AC$  if for some prime $p$ congruent to $3$ mod 4, $p$ divides $4n+1$ with odd multiplicity. 
  \item $T_n$ is of order 2 in $\AC$  otherwise.
  \end{itemize}
  
    In \cite{MyFirstPaper}, the author defines a $\rho$-invariant called $\rho^1$.  The set of  twist knots which are of order 2 or 4 in $\AC$ with $\rho^1(T_n)\neq 0$ is linearly independent in $\conc$ \cite[Corollary 4.4]{MyFirstPaper}.  If $L$ is a derivative of $T_n\#T_n$  
    then $\rho^1(T_n)$ has a bound in terms of a $\rho$-invariant of $M(L)$ corresponding to a homomorphism to an abelian group \cite[Theorem 5.7]{MyFirstPaper}.  In Section~\ref{application} we find bounds on this abelian $\rho$-invariant.  In doing so we prove the following theorem:

\begin{reptheorem}{linear indep done}
The infinite set containing all of the twist knots $T_n$ which are algebraically of order 2 is linearly independent in $\conc$ with the following $12$ possible exceptions:
$$
\begin{array}{rcl}
n&=&1, 
11, 
16, 
29, 
36, 
38, 
51, 55, 61, 
66, 83, 
101
.
\end{array}
$$
\end{reptheorem}

The $0$-twist knot is the unknot.  The $2$-twist knot is the stevedore knot, which is slice.  The $1$-twist knot, $T_1$ is the figure eight knot, which is of order 2 in the concordance group.  Thus $n=1$ is an exception to Theorem~\ref{linear indep done}.  In previous results, Casson and Gordon \cite{CG1} show that the algebraically slice twist knots (with the exception of the $0$ and $2$-twist knots) are not slice.  Jiang \cite{Ji1} shows that there is an infinite set of algebraically slice twist knots that are linearly independent in $\conc$.   Livingston and Naik \cite{LiN1} show that infinitely many of those twist knots that are algebraically of order 4 are not of finite order in $\conc$.  Tamulis \cite[Corollary 1.2]{Tamulis} finds an infinite linearly independent set of algebraic order 2 twist knots.   Cha \cite{polynomialSplittingOfCG} shows that no nontrivial linear combination of the twist knots (with the exception of the 0, 1, and 2-twist knots) is ribbon. Since the 2-fold branched cover of the $n$-twist knot is the Lens space $L(4n+1,2n)$, results of Lisca \cite{Lis1} show that the twist knots (with the same exceptions) are linearly independent in the \emph{smooth} concordance group.

It is worth noting that the result of Tamulis mentioned above is used in our proof of Theorem~\ref{linear indep done}.

\begin{figure}[h!]
\setlength{\unitlength}{1pt}
\begin{picture}(80,60)
\put(0,0){\includegraphics[width=.2\textwidth]{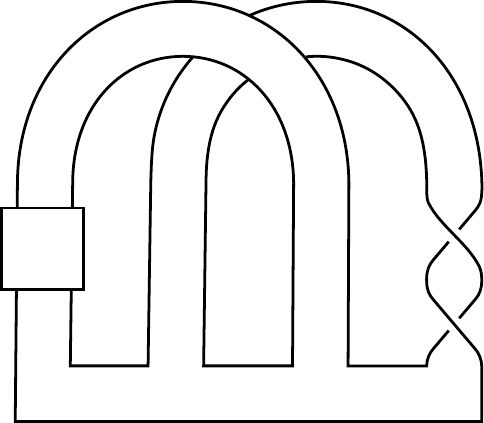}}
\put(3,23){$n$}
\end{picture}
\caption{The $n$-twist knot.}\label{fig:twist}
\end{figure}

\subsection{Signatures and $\rho$-invariants}\label{rho invt facts}

We recall the definitions of the unitary and $L^2$-$\rho$-invariants.  Each is given in terms of the twisted signatures of 4-manifolds. 

Let $W$ be a compact oriented connected 4-manifold.   There is a symmetric bilinear form $Q_W:H_2(W;\C)\times H_2(W;\C)\to \C$ called the intersection form.  In an abuse of notation we let $Q_W$ refer also to a matrix representing this form.  The signature of $W$ is defined to be the signature of $Q_W$:
$$
\sigma(W)=\sigma(Q_W).
$$
For concreteness, the signature of a
 complex Hermitian matrix is given by the number of positive eigenvalues minus the number of negative eigenvalues counted with multiplicity.  

A representation $\alpha:\pi_1(W)\to U(n)$ induces a $\C[\pi_1(W)]$-module structure on $\C^n$.  We will use each of $\C^n_\alpha$ and $V_\alpha$ to encode this module structure.  Let $\dim(\alpha) = n$ be the dimension of the representation.

For a unitary representation $\alpha$ of  the fundamental group of 
$W$, 
the $k$\th homology of $W$ twisted by $\alpha$ is denoted $H_k(W;V_\alpha)$.
  There is a  twisted intersection form,
$Q^\alpha_W:H_2(W;V_\alpha)\times H_2(W;V_\alpha)\to \cplx.$
This form is also Hermitian.  The signature of $W$ twisted by $\alpha$ is defined to be the signature of $Q_W^\alpha$:
$$
\sigma(W,\alpha)=\sigma(Q_W^\alpha).
$$

  In \cite{APS2}, Atiyah, Patodi and Singer associate to a 3-manifold $M$ and a unitary representation $\alpha$ of $\pi_1(M)$ an integer, $\rho(M,\alpha)$.  In this paper we make use of a reformulation due to Casson and Gordon \cite{CG1, CG2}.  
 
 \begin{definition}\label{definition of unitary rho}
  For compact connected oriented 3-manifolds $M_1,\dots, M_k$, suppose that $W$ is a compact connected oriented  4-manifold with  $\bdry W = \underset{j=1}{\overset{k}{\sqcup}}M_j$.  Let  $\alpha$ be a unitary representation of $\pi_1(W)$.  Let $\alpha_j = \alpha\circ (i_j)_*$ be the representation of $\pi_1(M_j)$ given by composition with the inclusion induced map $(i_j)_*:\pi_1(M_j)\to \pi_1(W)$.  Then
$
\displaystyle\sum_{j=1}^k\rho(M_j,\alpha_j)=\sigma(W,\alpha)-\dim(\alpha)\cdot\sigma(W).
$
\end{definition}


For a closed oriented 3-manifold $M$ with coefficient system $\phi:\pi_1(M)\to\Gamma$ the $L^2$-$\rho$-invariant is likewise given in terms of a coefficient system on a bounded 4-manifold, for example see \cite[Section 3]{Ha2}.  We treat this result as the definition of $\rho^{(2)}$.

\begin{definition}\label{defn of rho}
For $j=1,\dots k$ let $M_j$ be a 3-manifold and $\phi_j:\pi_1(M_j)\to\Gamma_j$ be a homomorphism.  If $W$ is a compact oriented connected 4-manifold with $\bdry W = \underset{j=1}{\overset{k}{\sqcup}} M_j$ and $\psi:\pi_1(W)\to \Lambda$ is a homomorphism such that such that for each $j$ there is a monomorphism $\theta_j:\Gamma_j\into \Lambda$ making the following diagram commute:
\begin{equation*}
 \begin{diagram}
\node{\pi_1(M_j)} \arrow{e,t}{\phi_j}
         \arrow{s,r}{(i_j)_*}
	     \node{\Gamma_j} \arrow{s,r,L}{\theta_j}\\
       \node{\pi_1(W)} \arrow{e,t}{\psi}
     \node{\Lambda} 
      \end{diagram}
\end{equation*}
      then 
      $\displaystyle\sum_{j=1}^k\rho^{(2)}(M_j,\phi_j) = \sigma^{(2)}(W,\psi) - \sigma(W)$ where $\sigma^{(2)}$ is the $L^2$-signature of $W$ twisted by the coefficient system $\psi$.
 \end{definition}
 


   Instead of recalling the definition of the $L^2$-signature, we state a powerful result of L\"uck and Schick which computes $L^2$-signatures corresponding to residually finite groups in terms of regular signatures of compact covers.    For a residually finite group $\Gamma$ a \textbf{resolution} for $\Gamma$ is a a nested sequence of finite index normal subgroups, 
$
\Lambda = \Lambda_0\ge \Lambda_1\ge\dots,
$
 with $\Lambda_0 \cap  \Lambda_1\cap\dots=0$.
   
   \begin{proposition}[Theorem 0.1 of \cite{approxL2Sign}]
   \label{LuckSchickThm}
   Suppose that $W$ is a compact oriented 4-manifold and $\phi:\pi_1(W)\to \Lambda$ is a homomorphism to a residually finite group with resolution
$
\{\Lambda_k\}.
$
 The composition $\phi_k:\pi_1(W)\to \Lambda \to \Lambda/\Lambda_k$ gives a homomorphism to a finite group.  This homomorphism corresponds to a compact cover $\widetilde{W_k}$ of $W$.   Then
\begin{equation}\label{residually finite formula}
\sigma^{(2)}(W,\phi) = \displaystyle\lim_{k\to\infty}{\dfrac{1}{|\Lambda/\Lambda_k|}\sigma(\widetilde{W_k})}
\end{equation}
where $\sigma(\widetilde{W_k})$ is the classical signature of the compact 4-manifold $\widetilde{W_k}$.
\end{proposition}

We essentially take Proposition~\ref{LuckSchickThm} as the definition of the $L^2$-signature, since every group that appears in this paper is residually finite.  Most are even finitely generated abelian groups.

Implicit in the above result is the fact that if $\phi:\pi_1(W)\to G$ is a homomorphism to a finite group, then $\sigma^{(2)}(W,\phi) = \displaystyle \dfrac{1}{|G|}\sigma(\widetilde W)$ where $\widetilde W$ is the compact cover corresponding to $\phi$.
The homology of the compact cover $\widetilde{W}$ can be expressed as the twisted homology of $W$ corresponding to the representation to the group of unitary transformations on the complex vector space $\cplx[G]\cong\cplx^{|G|}$:
$$
H_p(\widetilde{W};\cplx) = H_p(W;\cplx[G]).
$$
 
 Since  $G$ is finite, a classical fact of representation theory says that $\cplx[G]$ is a direct sum of representations of $G$.  That is, if 
  $X_G
  $
is the set of all irreducible unitary representations of $G$, then $\C[G]$ is isomorphic as a $\C[G]$-module to $\oplus_{\alpha\in X_G} (V_\alpha)^{\dim(\alpha)}$.  (For example see \cite[Section 6.2 Proposition 10]{Serre77}).  In particular, each $V_\alpha$ is projective and hence flat as a $\cplx[G]$-module.  Thus, 
$$
\begin{array}{rcl}
H_2(\widetilde {W};\C)
&\cong&
H_2(W;\C[G])
\cong
H_2\left(W;\oplus_{\alpha\in X_G} V_\alpha^{\dim(\alpha)}\right)
\\
&\cong&
\oplus_{\alpha\in X_G}H_2(W;V_\alpha)^{\dim(\alpha)}.
\end{array}
$$

Moreover, the intersection form on $H_2(\widetilde{W};\cplx)$ splits as the direct sum of the intersection forms on $H_2(W;V_\alpha)$ and
\begin{equation}
\label{unitary formula}
\sigma(\widetilde{W})=\sum_{\alpha\in X_G}\dim(\alpha)\cdot\sigma(W,\alpha\circ\phi).
\end{equation} 
Between Proposition~\ref{LuckSchickThm} and equation \eref{unitary formula} we see that the unitary signatures of a 4-manifold determine its finite and residually finite $L^2$-signatures:

\begin{proposition*}
Let $W$ be a 4-manifold and $\phi:\pi_1(W)\to \Lambda$  be a homomorphism to a residually finite group with resolution $\{\Lambda_k\}$.  Let $p_k:\Lambda\to \Lambda/\Lambda_k$ be the quotient map.  Let 
$X_k
$ 
be the set of all irreducible unitary representations of $\Lambda/\Lambda_k$.  The $L^2$ signature of $W$ twisted by $\phi$ is given in terms of unitary signatures of $W$:
$$
\sigma^{(2)}(W,\phi) = \lim_{k\to\infty}\dfrac{1}{|\Lambda/\Lambda_k|}
\sum_{\alpha\in X_k}\dim(\alpha)\cdot\sigma(W,\alpha\circ p_k\circ\phi)
.
$$
\end{proposition*}

\subsection{Organization of the paper}

Unitary and $L^2$-$\rho$-invariants of a 3-manifold are expressed in terms of bounded 4-manifolds.  In Section~\ref{4 manifold} for any link $L$ with zero pairwise linking numbers we construct a 4-manifold $W$ with $M(L)\subseteq \bdry W$.  We compute the one dimensional unitary signatures for almost all of the representations needed to prove Theorem~\ref{unitary case}.  One of the main difficulties in the completion of the proof is that the Cimasoni-Florens signature function, $\sigma_L$ is not defined on all of $\T^n$.  In Section~\ref{extend} we define the extension $\widehat{\sigma_L}:\T^n\to\Z$  and complete the proof of Theorem~\ref{unitary case}.  

In Section~\ref{L2 Proof} we  compute the $L^2$-signatures of the 4-manifold found in Section~\ref{4 manifold} and prove Theorem~\ref{L2 Rho}.

Let $F$ be a closed, oriented surface.  During the course of our arguments, the unitary and $L^2$-$\rho$-invariants of the 3-manifold $F\times S^1$ become relevant.  In Section~\ref{product rho} we prove that many such $\rho$-invariants vanish.  

\begin{proposition*}[Propositions~\ref{Unitary product} and \ref{L2 product}]
Let $F$ be a closed oriented surface, $\alpha:\pi_1(F\times S^1)\to U(1)$ be a representation with images in the roots of unity, and $\phi:\pi_1(F\times S^1)\to \Gamma$ be a homomorphism to a residually finite group.  Then $\rho(F\times S^1,\alpha) = \rho^{(2)}(F\times S^1, \phi)=0$.
\end{proposition*}
   
In Section~\ref{application} we use Theorem~\ref{L2 Rho} to provide bounds on the obstruction to linear dependence of \cite{MyFirstPaper}.  In doing so we prove the linear independence of the twist knots (Theorem~\ref{linear indep done}.)

\section{ The 4-manifold used to compute $\rho$-invariants of links}\label{4 manifold}
      
      We begin by stating the two main results which allow us to express the unitary $\rho$-invariants of $M(L)$ corresponding to one dimensional representations.

 \begin{lemma}\label{Build W}
For any $n$-component link $L$ with zero pairwise linking, there are closed oriented surfaces $F_1,\dots, F_n$ and a 4-manifold $W$ with $$\bdry W = M(L)\sqcup F_1\times S^1 \sqcup F_2\times S^1\sqcup \dots \sqcup F_n\times S^1$$ 
such that 
\begin{enumerate}
\item\label{H1} The inclusion induced map $H_1(M(L))\to H_1(W)$ is an isomorphism.
\item\label{signature} The classical signature of $W$ vanishes, that is $\sigma(W)=0$.
\item\label{unitary signature} For any $n$-tuple of roots of unity other than $1$, $\omega = (\omega_1,\dots, \omega_n)\in \T^n_\Q\cap \T^n_*$, and the representation $\alpha_\omega:H_1(W)\to U(1)$ sending the meridian of the $i$\th component of $L$ to $\omega_i$, $\sigma(W,\alpha_\omega) = \sigma_L(\omega)$.
\end{enumerate}
 \end{lemma}
 
  \begin{proposition}\label{Unitary product}
 For a closed oriented connected surface $F$, and any representation $\alpha:\pi_1(F\times S^1)\to U(1)$ with image in the roots of unity, $\rho(F\times S^1,\alpha)=0$
 \end{proposition}
 
 We give the proof of Lemma~\ref{Build W} at the end of this section and the proof of Proposition~\ref{Unitary product} in Section~\ref{product rho}.  
 
 Let $L$ be a link with zero pairwise linking and $\alpha:\pi_1(M(L))\to U(1)$ be a 
 representation.  
 Since $U(1)$ is abelian we may as well regard $\alpha$ as a representation of $H_1(M(L))$.  
 By conclusion~\eref{H1} of Lemma~\ref{Build W}, $\alpha$ extends over $W$ to a representation making the following diagram commute.
$$
\begin{diagram}
\node{\pi_1(M(L))}\arrow{s}\arrow{e}\node{H_1(M(L))}\arrow{s,l}{\cong}\arrow{e,t}{\alpha}\node{U(1)}\\
\node{\pi_1(W)}\arrow{e}\node{H_1(W)}\arrow{ne,t}{\alpha}
\end{diagram}
$$
This implies by Definition~\ref{definition of unitary rho} that 
\begin{equation}
\label{unitary rho eqn 1}
\sigma(W,\alpha)-\sigma(W) = \rho(M(L),\alpha) + \sum_{k=1}^n\rho(F_k\times S^1,\alpha).
\end{equation}

 Let $\omega = (\omega_1,\dots, \omega_n)\in\T^n_\Q$ and $\alpha = \alpha_\omega$ be the representation sending the meridian of the $i$\th component of $L$ to $\omega_i$. Conclusion~\eref{signature} of Lemma~\ref{Build W} 
 and Proposition~\ref{Unitary product} imply 
 that $\sigma(W)
 =\rho(F_k\times S^1,\alpha_\omega)
 =0$.  
If $\omega\in \T^n_*$ then by conclusion~\eref{unitary signature} of Lemma~\ref{Build W},  $\sigma(W,\alpha_\omega)= \sigma_L(\omega)$.   Making these substitutions reduces equation~\eref{unitary rho eqn 1} to the conclusion that $\sigma_L(\omega) = \rho(M(L),\alpha_\omega)$ for all $\omega\in \T^n_*\cap \T^n_\Q$.
 
For all $n$-tuples of roots of unity for which the  Cimasoni-Florens signature is  defined we see that it agrees with the unitary $\rho$-invariant of $M(L)$.  We have proven the following result:
 
 \begin{proposition}\label{unitary version take 1}
  If $L=L_1,\dots L_n$ is an $n$-component link with zero pairwise linking, $\omega = (\omega_1,\dots, \omega_n)\in \T^n_\Q$
   and $\alpha_\omega:H_1(M(L))\to U(1)$ sends the meridian of $L_i$ to $\omega_i$  then $\sigma(W,\alpha_\omega) = \rho(M(L),\alpha_\omega)$.  If additionally $\omega\in \T^n_*$, then
$\sigma_L(\omega) = \rho(M(L),\alpha_\omega).$
\end{proposition}

This almost completes the proof of Theorem~\ref{unitary case}.  It remains to analyze the case that $\omega$ is an $n$-tuple of roots of unity some of which are equal to $1$.   This is addressed in Section~\ref{extend}.  
 
 We close this section with a proof of Lemma~\ref{Build W}.
 
\begin{proof}[Proof of Lemma~\ref{Build W}]

For any $n$-component link $L = L_1,\dots, L_n$ with zero pairwise linking numbers, there are disjoint compact oriented embedded surfaces in the 4-ball, $F_1,\dots, F_n$ with $\bdry F_i = L_i$.  For example, let $F_i^0$ be a Seifert surface for $L_i$ disjoint from $L_j$ for $j\neq i$.  Notice that for $i\neq j$ $F_i^0\cap F_j^0\subseteq \inter(F_i^0) \cap \inter(F_j^0)$ is contained  in the interiors of $F_i^0$ and $F_j^0$.   Let $F_i$ be given by pushing the interior of $F_i^0$ into the 4-ball so that $\inter(F_i)$ and $\inter(F_j)$ lie at a different depths.  Then $F_i\cap F_j=\emptyset$ for all $i\neq j$ and we have the promised surfaces.

Let $E(F)$ be the 4-manifold given by cutting disjoint tubular neighborhoods of these surfaces out of the 4-ball.  We summarize some properties of $E(F)$.  The first two follow from a Mayer-Vietoris argument.  The third is a result of Cimason-Florens \cite[Theorem 6.1]{CimFlo}.
\begin{enumerate}
\item \label{EF H1}If $E(L)$ is the complement of a neighborhood of $L$ then $E(L)\subseteq \bdry E(F)$ and the inclusion induced map $H_1(E(L))\to H_1(E(F))$ is an isomorphism.
\item \label{EF H2}$H_2(E(L))=0$ so that $\sigma(W)=0$.
\item \label{EF sign}For any $\omega\in \T^n_\Q\cap \T^n_*$, and the representation $\alpha_\omega$ sending the meridian of $L_i$ to $\omega_i$, $\sigma(E(F),\alpha_\omega) = \sigma_L(\omega)$.
\end{enumerate} 

Let $W_0$ be the result of adding a 2-handle to the 4-ball along the zero framing of every component of $L$.  Let $D_i$ be the core of the 2-handle added to $L_i$.  Let $\widehat{F_i}$ be the closed surface given by gluing together $F_i$ and $D_i$.  Since the trivializations of  the normal bundles of each of $D_i$ and $F_i$ induce the zero framing of $L_i$ the surface $\widehat{F_i}$ has trivial normal bundle.  Let $W$ be given by removing disjoint tubular neighborhoods of the surfaces $\widehat{F_1}, \dots, \widehat{F_n}$ from $W_0$.

Notice that $\bdry W_0= M(L)$ is a component of $\bdry W$.  The remaining components are given by $ \widehat{F_i}\times S^1$.  Thus, the boundary of $W$ is as claimed in the lemma.

Alternately $W$ can be described by adding to $E(F)$ 
a total of $n$ 
``hollow 2-handles'' (copies of $B^2\times (B^2-\text{core})\cong B^2\times S^1\times [0,1]$) so that the longitude of each component of $L$ is identified with $\bdry B^2\times \{1\}\times\{1\}$  and the meridian is identified to $\{1\}\times S^1\times\{1\}$ in the corresponding hollow 2-handle.  Since the longitudes are  null-homologous in $E(F)$, it follows that $H_1(W)\cong \Z^n$ is generated by the meridians of the components of $L$.  This proves result \eref{H1} of the lemma.  

By a Mayer-Vietoris argument 
 $H_2(W)$ is free abelian and  has basis given by pushoffs of the components of $\widehat{F_i}$.  Since these are carried by the boundary, the intersection matrix of $W$ is the zero matrix and conclusion \eref{signature} follows.

Let $T_i$ be the subset of $\bdry E(F)$ to which the $i$\th hollow 2-handle is glued.  Notice that $T_i \cong S^1\times S^1\times[0,1]$.  
In order to see the conclusion~\eref{unitary signature} of the Lemma, consider the twisted Mayer-Vietoris exact sequence corresponding to the decomposition of  $W$ as the union of $E(F)$ with $n$ hollow 2-handles.  

 \begin{equation}\label{relate rho diagram}
\begin{array}{c}
\dots \to \iterate{i=1}{n}{\oplus}H_2(T_i;\cplx_{\alpha_\omega})\to 
\\
H_2(E(F);\cplx_{\alpha_\omega})\oplus\left(\iterate{i=1}{n}{\oplus}H_2(\text{hollow 2-handle};\cplx_{\alpha_\omega})\right)
\\ 
\to H_2(W;\cplx_{\alpha_\omega})\to\iterate{i=1}{n}{\oplus}H_1(T_i;\cplx_{\alpha_\omega})\to\dots
\end{array}
\end{equation}

The twisted chain complex for the $i^\text{th}$ hollow 2-handle (which is homotopy equivalent to  $S^1$) is chain homotopy equivalent to 
$$
\begin{diagram}
\node{\cplx}\arrow{e,b}{1-\omega_i}\node{\cplx,}
\end{diagram}
$$
while the twisted chain complex for $T_i$ (which is homotopy equivalent to $S^1\times S^1$) is chain homotopy equivalent to
$$
\begin{diagram}
\dgVERTPAD=7ex 
\dgHORIZPAD=1em 
\node{\cplx}\arrow[2]{e,b}{
\left(\begin{array}{c}0\\1-\omega_i\end{array}\right)
}\node[2]{\cplx^2}\arrow[2]{e,b}{
\left(\begin{array}{cc}1-\omega_i&0\end{array}\right)
}\node[2]{\cplx.}\\
\end{diagram}
$$
Since $\omega_i\neq1$ each of these chain complexes is acyclic and  $H_2(E(F);\C_{\alpha_\omega})\to H_2(W;\C_{\alpha_\omega})$ is an isomorphism.  This implies that $W$ and $E(F)$ have  isomorphic twisted intersection forms and so have the same signature:   $\sigma(W,\alpha_\omega) = \sigma(E(F),\alpha_\omega)$. By \cite[Theorem 6.1]{CimFlo} $\sigma(E(F),\alpha_\omega)=\sigma_L(\omega)$.

\end{proof}

\section{Extending $\sigma_L$ over $\T^n$.}\label{extend}

In this section we provide an extension of the Cimasoni-Florens signature functon.  This extension is given in terms of the  signature of a new link:

\begin{definition}
For any link $L$ let $L^\pm$ be the link gotten by replacing each component $L_i$ of $L$ with two parallel copies of $L_i$ with opposite orientations and zero linking numbers.
  $L_i^+$  denotes the copy with the same orientation as $L_i$ and $L_i^-$  denotes the copy with the opposite orientation.  Order the components of $L^\pm$ by $(L^\pm)_i=L^+_i$ and $(L^\pm)_{n+i}=L^-_i$ for $i\le n$.
\end{definition}

\begin{definition}\label{define extension}
     Let $z:\T^n\to \T^{2n}_*$ be given by $z(\omega)=(z_1(\omega),\dots, z_{2n}(\omega))\in \T^{2n}_*$ where 
     $$
z_i(\omega)=\left\{\begin{array}{cl}
-\omega_i\sqrt{-1} &\text{ if } \omega_i \neq \sqrt{-1}
\\
-\sqrt{-1} &\text{ if }\omega_i = \sqrt{-1}
\end{array}
\right.
\text{ and }
z_{n+i}(\omega)=
\left\{\begin{array}{cl}
\sqrt{-1}&\text{ if } \omega_i \neq -\sqrt{-1}
\\
-1 &\text{ if }\omega_i = \sqrt{-1}.
\end{array}
\right.
$$
for $i=1,\dots n$.  Define
$\widehat{\sigma}_L(\omega) := \sigma_{L^\pm}(z(\omega))$.
\end{definition}

The important properties of $z(\omega)$ are that $z_i(\omega)\cdot z_{n+i}(\omega) = \omega_i$ for $i=1,\dots, n$ and $z_i(\omega)\neq 1$ for all $i=1,\dots, 2n$.

\begin{lemma}\label{zero surgery lemma}
For any $n$ component link $L$ the zero surgery on $L^\pm$ is differmorphic to the zero surgery on the split union $L\cup U$ where $U=U_1\cup\dots\cup U_n$ is the $n$-component unlink.  The diffeomorphism sends the meridian of $L_i^-$ to the meridian of  $U_i$ and sends the meridian of $L_i^+$ to a curve homologous to the sum of the meridian of $L_i$ with the meridian of $U_i$.
\end{lemma}
\begin{proof}

  By sliding $L_i^-$ over the $0$-framing of $L_i^+$ we replace each $L_i^-$ with a curve which bounds a disk in the complement of all other components.   The resulting link is the split union of $L$ with the unlink, $L\cup U$.  

By Kirby calculus  (See for example \cite[Section 5.1]{KirbyCalc}), $M(L^{\pm})$ is diffeomorphic to $M(L\cup U)$.  Moreover the diffeomorphism sends the meridian of $L_i^-$ to the meridian of $U_i$ and the meridian of 
$L_i^+$ to a curve homologous to the sum of the meridian of $U_i$ with the meridian of $L_i$.

\end{proof}

Since $M(L\cup U)\cong M(L^\pm)$, they must have the same $\rho$-invariants.  Additionally, taking the disjoint union with the unlink does not change $\rho$-invariants.  Thus, the $\rho$-invariants of $M(L)$ agree with the $\rho$-invariants of $M(L^\pm)$.  The following lemma makes this precise.

\begin{lemma}\label{change link}
Let $\omega = (\omega_1,\dots, \omega_n)\in \T^n$.  Then $\rho(M(L),\alpha_\omega) = \rho(M(L^\pm),\alpha_{z(\omega)})$.
\end{lemma}
\begin{proof}

Let $\mu_i$ be the meridian of $L_i$, $m_i$ be the meridian of $U_i$ and $\mu_i^+$ and $\mu_i^-$ be the meridians of $L_i^+$ and $L_i^-$ respectively.  Let $\beta:H_1(M(L\cup U))\to U(1)$ be given by 
$$
\beta(\mu_i) =  \omega_i
\text{ and }
\beta(m_i) =  (z_{n+i}(\omega))^{-1}. 
$$ 
Let $\phi: M(L^\pm)\to M(L\cup U)$ be the diffeomorphism of Lemma~\ref{zero surgery lemma} and $\alpha':H_1(M(L^\pm))\to U(1)$ be given by $\alpha' = \beta\circ \phi_*$.  Since $\rho$-invariants are diffeomorphism invariants, it follows that $\rho(M(L^\pm),\alpha') = \rho(M(L\cup U),\beta)$.
Annlysis of this composition reveals 
$$
\begin{array}{rcl}
\alpha'(\mu_i^+) &=& \beta(\phi_*(\mu_i^+)) 
= \beta(\mu_i+m_i) 
= \omega_i / z_{n+i}(\omega) = z_i(\omega)
\\
\alpha'(\mu_i^-) &=& \beta(\phi_*(\mu_i^-)) = \beta(m_i) = z_{n+i}(\omega),
\end{array}
$$
so that $\alpha' = \alpha_{z(\omega)}$.

Notice that $M(L\cup U)$ is diffeomorphic to the connected sum of $M(L)$ with $n$ copies of $S^2\times S^1$ and that the restriction of $\beta$ to the $M(L)$-summand is precisely $\alpha_\omega$.  Since $\rho$-invariants are additive under connected sums of 3-manifolds, 
$\rho(M(L\cup U),\beta)$ 
is given by the sum of $\rho(M(L),\alpha_\omega)$ with a sum of $\rho$-invariants of $S^2\times S^1$.  
$$
\rho(M(L^\pm),\alpha_{z(\omega)})=\rho(M(L\cup U),\beta) = \rho(M(L),\alpha_\omega)+\sum_k\rho(S^2\times S^1,\beta).
$$

 The $\rho$-invariants of $S^2\times S^1$ vanish (for example by Lemma~\ref{unitary case}).  Making this substitution,
$
\rho(M(L^\pm),\alpha_{z(\omega)}) 
= \rho(M(L),\alpha_\omega),
$ as we claimed.

\end{proof}

If $\omega\in \T^n_\Q$, then $z(\omega)\in \T^n_\Q\cap\T^n_*$.  Proposition~\ref{unitary version take 1} applies:
$$
\begin{array}{rcll}
\rho(M(L)),\alpha_\omega) 
&=& 
\rho(M(L^{\pm}), \alpha_{z(\omega)})&\text{By Lemma~\ref{change link}}
\\&=&
\sigma_{L^\pm}(z(\omega))&\text{By Proposition~\ref{unitary version take 1}}
\\&=&
\widehat{\sigma_L}(\omega)&\text{By Definition~\ref{define extension}}.
\end{array}
$$  
Thus, we have proven the first main theorem of this paper:



\begin{theorem}\label{unitary case}
Let $L = L_1,\dots, L_n$ be a link with zero pairwise linking.   
\begin{enumerate} 
\item \label{gives rho} For $\omega\in \T^n_\Q$ and the representation $\alpha_{\omega}:H_1(M(L))\to U(1)$ and sending $\mu_i\mapsto \omega_i$, $\rho(M(L),\alpha_{\omega}) = \widehat{\sigma}_L(\omega)$.
\item \label{agree} $\sigma_L$ and $\sigma_{L^\pm}\circ z_\gamma$ agree on $\T^n_\Q\cap \T^n_*$.
\end{enumerate}
\end{theorem}

We move on to an analysis of the continuity of $\widehat{\sigma_L}$ generalizing \cite[Theorem 4.1]{CimFlo}.  Recall that $R\subseteq \T^n$ is an \textbf{algebraic subset} if $R$ is the intersection of $\T^n$ with the zero loci of some polynomials.

\begin{proposition}\label{continuity1}
For any link $L$, there is a sequence of real algebraic sets $\T^n = R_0 \supseteq R_1\supseteq \dots \supseteq R_m=\emptyset$ such that $\widehat\sigma_L$ is locally constant on $R_j-R_{j+1}$.
\end{proposition}

\begin{proof}

First notice that a map to $\Z$ being locally constant is the same as that map being continuous.


  Results of Cimasoni-Florens \cite[Theorem 4.1]{CimFlo} provide a sequence of algebraic sets $\T^{2n}=\Sigma_0 \supseteq \dots \supseteq \Sigma_k  = \emptyset$ satisfying that $\sigma_{L^{\pm}}$ is continuous on $\T^{2n}_* \cap (\Sigma_i-\Sigma_{i+1})$.
For $0\le j\le n$ let 
$$
S_j:=\{(\omega_1,\dots, \omega_n)\in \T^n \text{ such that }\omega_1 = \dots = \omega_j=\sqrt{-1}\}
$$
so that $S_0=\T^n$.  Define $S_{n+1}=\emptyset$.
Observe that $z$ is continuous on $S_j - S_{j+1}$.  Indeed, restricted to $S_j - S_{j+1}$ $z$ is given by a polynomial map.  Call this polynomial map $Z_j$.  A composition of continuous maps is continuous so that on 
$
(S_j - S_{j+1})\cap z^{-1}[\Sigma_i-\Sigma_{i+1}]
$
the signature function $\widehat{\sigma_L} = \sigma_{L^\pm}\circ z$ is continuous.

  
Since $z$ and $Z_j$ agree on $(S_j - S_{j+1})$, 
$$
(S_j - S_{j+1})\cap z^{-1}[\Sigma_i] = (S_j - S_{j+1})\cap Z_j^{-1}[\Sigma_i].
$$
  Since $\Sigma_i$ is an algebraic set and $Z_j$ is given by polynomials, $Z_j^{-1}[\Sigma_i]$ is an algebraic set.  

Let $X_{i,j} = (S_{j}\cap Z_j^{-1}[\Sigma_i])\cup S_{j+1}$.  Then 
$$
\begin{array}{rcl}
X_{i,j} - X_{i+1,j} &=& (S_j-S_{j+1})\cap Z_j^{-1}[\Sigma_i] - (S_j-S_{j+1})\cap Z_j^{-1}[\Sigma_{i+1}]
\\&=&
(S_j-S_{j+1})\cap(Z_j^{-1}[\Sigma_i - \Sigma_{i+1}])
\\&=&
(S_j-S_{j+1})\cap(z^{-1}[\Sigma_i - \Sigma_{i+1}])
\end{array}
$$
on which we have already seen that $\widehat{\sigma_L}$ is continuous.  Additionally, since for all $j$, $Z_j^{-1}[\Sigma_k]=\emptyset$ and  $Z_j^{-1}[\Sigma_0]=\T^{n}$, it follows that $X_{k,j} = X_{0,j+1}=S_{j+1}$.
Thus,
$$
\begin{array}{rcccccccccccl}
\T^{2n} &=& X_{0,0}&\supseteq& X_{1,0}&\supseteq&\dots&\supseteq&X_{k-1,0}&\supseteq& X_{k,0}\\
X_{k,0}&=&X_{0,1}&\supseteq& X_{1,1}&\supseteq&\dots&\supseteq&X_{k-1,1}&\supseteq& X_{k,1}
\\
\vdots\\
X_{k,n-1}&=&X_{0,n}&\supseteq& X_{1,n}&\supseteq&\dots&\supseteq&X_{k-1,n}&\supseteq& X_{k,n}&=&\emptyset.
\end{array}
$$
Defining $R_{i+j\cdot k}:=X_{i,j}$ for $0\le i<k$ and $0\le j \le n$ provides the desired descending sequence of  algebraic sets.

\end{proof}

We close with a corollary regarding the integrability of $\widehat\sigma_L$ which will be relevant in the next section.

\begin{corollary}\label{continuity}
Let $L$ be an $n$-component  link with zero pairwise linking numbers and $A = \langle g_1,\dots, g_n| r_1,\dots r_m\rangle$ be an abelian group.  Then  $\widehat{\sigma_L}$ is continuous on $\T_A\subseteq \T^n$ away from an algebraic set of  codimension at least one.  In particular $\widehat{\sigma_L}$ is Riemann integrable on $\T_A$.
\end{corollary}
\begin{proof}

We begin by noticing that $\T_A$ is an algebraic set of dimension equal to the rank of $A$.  Let $U_1,\dots, U_a$ be the irreducible components of $\T_A$.  

for each $U_j$, let $k_j$ be the first index for which the set $R_{k_j}$ of Proposition~\ref{continuity1} does not contain $U_j$.  In this case $R_{k_j}\cap U_j$ is a proper algebraic subset of the irreducible algebraic variety $U_j$.  By a basic fact of algebraic geometry, for example \cite[Theorem 1.6.1]{Shafarevich1994},  $R_{k_j}\cap U_j$ is codimension at least one in $U_j$.

Thus $\widehat\sigma_L$ is continuous on $U_j$ away from the algebraic subset of codimension at least one subset, $R_{k_j}\cap U_j$.  This implies that $\widehat\sigma_L$ is continuous on all of $\T_A$ away from the codimension at least one algebraic  subset 
$
\underset{j=1}{\overset{a}\cup}(R_{k_j}\cap U_j).
$
\end{proof}

\section{Computing $L^2$-$\rho$-invariants of zero surgery.}\label{L2 Proof}

Let $L$ be a link with zero pairwise linking and $\phi:\pi_1(M(L))\to A$ be a homomorphism to an abelian group.   In this section we compute  $\rho^{(2)}(M(L),\phi)$.  We prove the following theorem:

\begin{theorem}\label{L2 Rho}
Let $A=\langle g_1\dots g_n|r_1\dots r_m\rangle$ be a rank-$R$ abelian group.  Let $L$ be an $n$-component link with zero pairwise linking numbers.  Let $\phi:\pi_1(M(L))\to A$ be given by sending $\mu_i$ to $g_i$, then 
\begin{equation*}\rho^{(2)}(M(L),\phi)=\frac{1}{\lambda(\T_A)}\int_{\T_A}\widehat\sigma_L(\omega)d\lambda(\omega)\end{equation*}
where $\lambda$ is $R$-dimensional Hausdorff measure on $\T_A$.
\end{theorem}

\begin{proof}

Let $W$ be the 4-manifold of Lemma~\ref{Build W}.  Since $A$ is abelian we regard $\phi$ as a having domain $H_1(M(L))$.  Since $H_1(M(L))\to H_1(W)$ is an isomorphism, $\phi$ extends over $H_1(W)$ making the following diagram commute
$$
\begin{diagram}
\node{\pi_1(M(L))}\arrow{s}\arrow{e}\node{H_1(M(L))}\arrow{s,l}{\cong}\arrow{e,t}{\phi}\node{A}\\
\node{\pi_1(W)}\arrow{e}\node{H_1(W).}\arrow{ne,t}{\phi}
\end{diagram}
$$
By Definition~\ref{defn of rho}, 
\begin{equation}\label{towards goal 1}
\sigma^{(2)}(W,\phi)-\sigma(W) = \rho^{(2)}(M(L),\phi) - \sum_{k=1}^n\rho^{(2)}(F_k\times S^1,\phi).
\end{equation}

In Section~\ref{product rho} we prove that many $L^2$-$\rho$-invariants of $F_k\times S^1$ vanish.  

\begin{proposition}\label{L2 product}
For a closed oriented surface $F$ and any homomorphism to a residually finite group (for example a finitely generated abelian group) $\phi:\pi_1(F\times S^1)\to H$, $\rho^{(2)}(F\times S^1,\phi)=0$.
\end{proposition}
By claim \eref{signature} of Theorem~\ref{Build W}, $\sigma(W)=0$.  Proposition~\ref{L2 product} now  reduces equation~\eref{towards goal 1}  to the conclusion that  $\sigma^{(2)}(W,\phi) = \rho^{(2)}(M(L),\phi)$.  In order to prove Theorem~\ref{L2 Rho} it suffices to compute the $L^2$-signatures of $W$.  The remainder of the proof falls into one of two cases, depending if $A$ is an infinite or finite abelian group.


\emph{Case 1:  $A$ is finite.}  In this case, 
\begin{equation}\label{finite step 1}\sigma^{(2)}(W,\phi) = \dfrac{1}{|A|}\sigma(\widetilde W)
\end{equation}
 where $\widetilde W$ is the compact cover corresponding to $\phi$.  By equation~\eref{unitary formula}, since $\T_A$ parametrizes all unitary representations of $A$ and $|\T_A| = |A|$:
 \begin{equation}\label{finite step 2}
 \sigma^{(2)}(W,\phi) = \dfrac{1}{|A|}\sigma(\widetilde W) = \dfrac{1}{|\T_A|}\sum_{\omega\in \T_A}\sigma(W,\alpha_\omega)
 \end{equation}

By Proposition~\ref{unitary version take 1}, $\sigma(W,\alpha_\omega) = \rho(M(L),\alpha_{\omega})$.  By Theorem~\ref{unitary case}, for all $\omega\in \T^n_\Q$, and in particular all $\omega\in \T_A$, $\rho(M(L),\alpha_{\omega}) = \widehat{\sigma_L}(\omega)$.  Recall that when $A$ is a finite abelian group, i.e. when the rank of $A$ is $0$,  $\lambda$ is $0$-dimensional Hausdorff measure (counting measure), $\lambda(\T_A) = |\T_A|$ and integration against $\lambda$ is summation.  Making these substitutions proves  that if $A$ is finite then
\begin{equation*}\sigma^{(2)}(W,\phi)=\frac{1}{\lambda(\T_A)}\int_{\T_A}\widehat\sigma_L(\omega)d\lambda(\omega)\end{equation*}
as claimed.

\emph{Case 2:  A is infinite.}  Being a  finitely generated abelian group, $A$ is residually finite.  Indeed, let $t\ge2$ be an annihilator of the torsion subgroup of $A$ and 
 set $A_k:=t^k\cdot A$.  Since $t$ annihilates the torsion subgroup, $A_1$ is contained in the torsion-free part and so is free abelian, $A_1\cong \Z^R$.  Since $\Z^R$ is not divisible as a $\Z$-module, and $A_k=t^{k-1}\cdot A_1$, it follows that $A_1\cap A_2\cap\dots =0$.  $A/A_k$ is finitely generated, abelian and  torsion.  Any such group is finite.  Thus, $A_k$ is finite index in $A$.

 Proposition~\ref{LuckSchickThm} and the proof in the finite case together imply that
\begin{equation}\label{L2CaseAlmostDone}
\displaystyle
\sigma^{(2)}(W,\phi) 
= 
\lim_{k\to\infty}\sigma^{(2)}(W,p_k\circ \phi) 
=
\lim_{k\to\infty}\sum_{\omega\in \T_{A/A_k}}\dfrac{\widehat\sigma_L(\omega)}{|\T_{A/A_k}|}
\end{equation}
where $p_k:A\to A/A_k$ is the quotient map.

It remains to realize the sum on the left hand side of \eref{L2CaseAlmostDone} as an integral.  For each $\omega\in \T_{A/A_k}$ let $U_\omega$ be the collection of points in $\T_A$ at least as close to $\omega$ as to any other point in $\T_{A/A_k}$:
$$
U_{\omega} = \{z\in\T_{A}\text{ such that for all }\omega'\in \T_{A/A_k},  |z-\omega|\le|z-\omega'| \}.
$$
  If $A$ is an infinite group, then 
  for 
  $\omega\neq \omega'$, $U_\omega \cap U_{\omega'}$ is codimension at least 1 in $\T_A$ (and so has zero $R$-dimensional Hausdorff measure).  Notice that multiplication by $\omega^{-1}\omega'$ sends $U_\omega$ to $U_{\omega'}$ so that all of these sets all have the same normalized Hausdorff measure.  By additivity this implies that $\lambda(U_{\omega}) = \dfrac{\lambda(\T_A)}{|\T_{A/A_k}|}$. Making this substitution reduces equation~\eref{L2CaseAlmostDone} to
  \begin{equation}\label{L2CaseAlmostDone2}
  \begin{array}{rcl}
\sigma^{(2)}(W,\phi) 
&=&
 \displaystyle\lim_{k\to\infty}\displaystyle\sum_{\omega\in \T_{A/A_k}}\dfrac{\widehat\sigma_L(\omega)\cdot \lambda(U_\omega)}{\lambda(\T_A)}
\\&=&
\dfrac{1}{\lambda(\T_A)}\displaystyle\lim_{k\to\infty}\displaystyle\sum_{\omega\in \T_{A/A_k}} \widehat\sigma_L(\omega)\cdot \lambda(U_\omega)
\end{array}
\end{equation}
  
  The collection $\{U_\omega\}$ forms a partition of $\T_A$ and
$
\displaystyle
\sum_{\omega\in \T_{A/A_k}} \widehat\sigma_L(\omega)\cdot \lambda(U_\omega)
$
is a Riemann sum for the integral of $\sigma_L$ over $\T_A$ with respect to $\lambda$.  According to Corollary~\ref{continuity}  $\widehat\sigma_L$ is  Reimann integrable on $\T_A$ so that  Riemann sums converge to the integral:  
$$
\sigma^{(2)}(W,\phi)=\dfrac{1}{\lambda(\T_A)}\lim_{k\to\infty}\sum_{\omega\in \T_{A/A_k}} \widehat\sigma_L(\omega)\cdot \lambda(U_\omega)
=
\frac{1}{\lambda(\T_A)}\int_{\T_A}\widehat\sigma_L(\omega)d\lambda(\omega).
$$

Thus, for any finitely generated abelian group $A$,  
$$
\rho^{(2)}(M(L),\phi) = \sigma^{(2)}(W,\phi)
=
\frac{1}{\lambda(\T_A)}\int_{\T_A}\widehat\sigma_L(\omega)d\lambda(\omega).
$$
\end{proof}


\section{The $\rho$ invariants of product 3-manifolds}\label{product rho}

The results of the preceding sections were dependent on Propositions \ref{Unitary product} and \ref{L2 product} which claim that the  $\rho$-invariants of $F\times S^1$ vanish for every closed oriented surface $F$.  In this section we prove these results.

\begin{repproposition}{Unitary product}
 For a closed oriented connected surface $F$, and any representation $\alpha:\pi_1(F\times S^1)\to U(1)$ with image in the roots of unity, $\rho(F\times S^1,\alpha)=0$.
 \end{repproposition}
 
\begin{proof}

In the case $F=S^2$, $\alpha$ extends over $S^1\times B^3$ and, $\rho(S^2\times S^1, \alpha)=\sigma(S^1\times B^3,\alpha)-\sigma(S^1\times B^3)$.  Since $S^1\times B^3$ is a homotopy circle, $\sigma(S^1\times B^3,\alpha) = \sigma(S^1\times B^3)=0$.

Just as in previous proofs, since $U(1)$ is abelian, we regard $\alpha$ as a representation of $H_1(F\times S^1)$.  If $F$ is not the 2-sphere then let $\{a_1,b_1,\dots a_g,b_g\}$ be a symplectic basis for $F$.  Consider the factorization
$$
\Z^2\cong\langle a_i,b_i\rangle\to H_1(F)\to H_1(F\times S^1)\overset{\alpha}{\to} U(1).
$$
Since every finitely generated subgroup of the roots of unity is cyclic, this factors as
$$\Z^2\cong\langle a_i,b_i\rangle \to \Z\to U(1)$$
and we see that there must be a nontrivial primitive linear combination $\alpha_i\cdot a_i+\beta_i\cdot b_i$ which $\alpha$ sends to the identity element.  Since it is primitive, $\alpha_i\cdot a_i+\beta_i\cdot b_i$ is represented by a simple closed curve $\gamma_i$.  Moreover, for $i\neq j$ $\gamma_i$ and $\gamma_j$ do not intersect  and $\gamma_1\cup\dots \cup \gamma_g$ does not separate.  Let $X$ be the 3-dimensional  handle body gotten by adding 2-handles to each $\gamma_i$.  Since $\ker(H_1(F\times S^1)\onto H_1(X\times S^1))\subseteq \ker(\alpha)$,  $\alpha$ extends over $X\times S^1$:
$$
\begin{diagram}
\node{H_1(F\times S^1)}\arrow{e,t}{\alpha}\arrow{s}\node {U(1)}
\\
\node{H_1(X\times S^1)}\arrow{ne,t}{\overline{\alpha}}
\end{diagram}
$$
and $\rho(F\times S^1,\alpha) = \sigma(X\times S^1,\overline{\alpha}) - \sigma(X\times S^1)$.  Since $X\times S^1$ deformation retracts to a subset of $\bdry (X\times S^1)$, the inclusion induced maps $H_2(\bdry (X\times S^1))\to H_2(X\times S^1)$ and $H_2(\bdry (X\times S^1);\cplx_{\overline{\alpha}})\to H_2(X\times S^1;\cplx_{\overline{\alpha}})$ are each surjections so that $\sigma(X\times S^1) = \sigma(X\times S^1,{\overline{\alpha}})=0$.

\end{proof}
 
 \begin{repproposition}{L2 product}
For a closed oriented surface $F$ and any homomorphism to a residually finite group
 $\phi:\pi_1(F\times S^1)\to H$, $\rho^{(2)}(F\times S^1,\phi)=0$.
\end{repproposition}

\begin{proof}[Proof of Proposition~\ref{L2 product}]

The proof proceeds in three steps: 
\begin{itemize}
\item \emph{Step 1:} First we give the proof when $H$ is a  finite cyclic group.  
\item \emph{Step 2:} Next we generalize the proof to any finite group.  
\item \emph{Step 3:} Finally we use the approximation theorem of L\"uck-Schick (Proposition~\ref{LuckSchickThm}) to do the computation in the residually finite case.
\end{itemize}
\emph{Step 1:}
In the case that $H$ is finite cyclic  then similarly to the proof of Proposition~\ref{Unitary product} there is a handle body $X$ bounded by $F$ such that $\phi$ extends to a map $\overline{\phi}:X\times S^1\to H$.  
 Thus, 
 $$\rho^{(2)}(F\times S^1,\phi) = \sigma^{(2)}(X\times S^1,\overline{\phi}) - \sigma(X\times S^1).$$
   Since $\overline{\phi}$ is a homomorphism to a finite group, it induces a compact cover $\widetilde{X\times S^1}$ and $\sigma^{(2)}(X\times S^1,\overline{\phi}) = \dfrac{1}{|H|}\sigma(\widetilde{X\times S^1})$.  Since $X\times S^1$ deformations retracts to a subset of $\bdry (X\times S^1)$, the inclusion induced maps $H_2(\bdry (X\times S^1))\to H_2(X\times S^1)$ and $H_2(\widetilde{\bdry (X\times S^1)})\to H_2(\widetilde{X\times S^1})$ are each surjections so that $\sigma(X\times S^1) = \sigma(\widetilde{X\times S^1})=0$. This completes the proof in the finite cyclic case. 

Before we move on we need a property of $\rho^{(2)}$-invariants.  Namely, for a 3-manifold $M$, groups $A$ and $B$, a homomorphism $\phi:\pi_1(M)\to A$ and a monomorphism $\psi:A\into B$,  $\rho^{(2)}(M,\phi)=\rho^{(2)}(M,\psi\circ\phi)$.  (See \cite[Proposition 5.13]{whitneytowers}.)  Thus, by replacing $H$ with $\im(\phi)$, we may assume that $\phi$ is onto.

Let $s$ denote the generator of $\pi_1(S^1)$ so that $\pi_1(F\times S^1)=\pi_1(F) \oplus \langle s\rangle$.   Let $\phi(s)=h$.  Since $s$ is central in $\pi_1(F\times S^1)$ and $\phi$ is onto, $h$ is central in $H$.  That is, $h$ commutes with every element of $H$.

\emph{Step 1:}  Suppose that $H$ is a finite group.  Then $h^n$ is trivial in $H$ for some $n$.  Let $P$ be the $n$ times punctured sphere.  Then $\pi_1(P) = \langle s_1\dots s_n| s_1s_2\dots s_n=1\rangle$ where each $s_k$ is given by a boundary component of $P$.  Consider the homomorphism 
$$
\overline \phi:\pi_1(F\times P)\cong\pi_1(F)\oplus\langle s_1\dots s_n| s_1s_2\dots s_n=1\rangle\to H
$$
given by $
\overline{\phi}(f)=\phi(f)\text{ if }f\in \pi_1(F)$
and 
$\overline\phi(s_i)=h \text{ for }i=1,\dots, n$.
  Notice that $\bdry(F\times P)=F\times \bdry P$ is given by $n$ copies of $F\times S^1$.  On each of these boundary components, $\overline \phi$ restricts to $\phi$.  Thus, $n\rho(F\times S^1, \phi) = \sigma^{(2)}(F\times P, \overline{\phi})-\sigma(F\times P) $.  By a K\"unneth theorem based argument, the map $H_2(F\times \bdry P)\to H_2(F\times P)$ is surjective and 
\begin{equation}\label{FbyPSign}
\sigma(F\times P)=0.
\end{equation}

Since $H$ is finite, the $L^2$-signature and the signature of the $\overline{\phi}$-cover agree,
$$\sigma^{(2)}(F\times P,\overline{\phi})=\frac{1}{|H|}\sigma(\widetilde{F\times P}_{\overline{\phi}}).$$  In order to compute $\sigma(\widetilde{F\times P}_{\overline{\phi}})$, we first consider the cover of $F\times P$ corresponding to the composition 
$\psi:\pi_1(F\times P)\overset{\overline \phi}{\to} H\to \frac{H}{\langle h \rangle}$.  Again, $h$ is central in $H$, so the cyclic subgroup it generates is normal in $H$.  Since $\psi$ is trivial on $\pi_1(P)$, the corresponding cover is given by the product of a (compact) cover of $F$ with $P$: $\widetilde{F\times P}_\psi = \widetilde{F}_\psi\times P$.  A compact cover of a closed oriented surface is still a closed oriented surface.

Since $\psi$ factors through $\phi$, the covering map corresponding to $\phi$ factors through the covering map corresponding to $\psi$.  Consider the resulting tower of covers.

\begin{equation*}
\begin{diagram}
\node{\widetilde{F\times P}_{\overline \phi}}\arrow[2]{s}\arrow{se,t}{\alpha}\\
\node[2]{\widetilde{F}_\psi\times P}\arrow{sw}\\
\node{F\times P}
\end{diagram}
\end{equation*}
The group of deck translations of $\alpha$ is isomorphic to 
$$
\Gamma:=
\frac{\pi_1(\widetilde{F}_\psi\times P)}{\alpha_*[\pi_1(\widetilde{F\times P}_{\overline \phi}]}) 
\cong 
\frac{\ker(\psi)}{\ker(\overline{\phi})}
\cong
\overline{\phi}[\ker(\psi)]
=
\ker\left(H \to \frac{H}{\langle h\rangle}\right)
=
\langle h\rangle\cong \Z_n.
$$

Thus, $\widetilde{F\times P}_{\overline \phi}$ is a finite cyclic cover of $F' \times P$ where $F':=\widetilde{F}_\psi$.  Let $\phi'$ be the homomorphism from $\pi_1(F'\times P)$ to $\Gamma\cong \Z_n$. 

We summarize what we have shown.  First,
$$n\rho^{(2)}(F\times S^1,\phi)=\sigma^{(2)}(F\times P,\overline{\phi}) - \sigma(F\times P) = \sigma^{(2)}(F\times P,\overline{\phi}).$$
  Since $\overline{\phi}$ maps to the finite group $H$, 
 $$
n\rho^{(2)}(F\times S^1,\phi) = \sigma^{(2)}(F\times P,\overline{\phi}) = \dfrac{1}{|H|}\sigma(\widetilde{F\times P}_{\overline\phi}).
 $$
 Since $\widetilde{F\times P}_{\overline\phi}$ is also the $n$-fold cyclic cover of $F'\times P$ corresponding to the map $\phi'$.
  $$
n\cdot|H|\cdot\rho^{(2)}(F\times S^1,\phi) = \sigma(\widetilde{F\times P}_{\overline{\phi}}) = \sigma(\widetilde{F'\times P}_{{\phi'}}) = n\sigma^{(2)}(F'\times P,\phi').
 $$
 On the other hand, using Definition~\ref{defn of rho} and the result that $\sigma(F'\times P)=0$,
 $$
|H|\rho^{(2)}(F\times S^1,\phi) =  \sigma^{(2)}(F'\times P,\phi') 
= n\rho^{(2)}(F'\times S^1,\phi').
 $$
 Thus, we see that $\rho^{(2)}(F\times S^1,\phi) = \dfrac{n}{|H|}\rho^{(2)}(F'\times S^1,\phi').$  Since $\phi'$ is a map to the finite cyclic group $\Gamma\cong \Z_n$, the first step of the proof gives us that $\rho^{(2)}(F'\times S^1,\phi')$.  This completes the proof in the case that $H$ is a finite group.

\textit{Step 3: }Finally we perform the general proof in the case that $H$ is any residually finite group.

Let $\Sigma$ be a once punctured torus and $\{a, b\}$ be a symplectic basis for $\Sigma$.  Let $E=\pi_1(\Sigma)$ be the free group on $a$ and $b$.  Consider the 4-manifold $F\times \Sigma$ and the group, $$G=\frac{H\oplus E}{h=[a,b]}$$
where $[a,b]=aba^{-1}b^{-1}$.  Denote the obvious map $\pi_1(F\times \Sigma)=\pi_1(F)\oplus E \to G$ by $\overline{\phi}$.  On the boundary, $\overline{\phi}$ factors through $H$.  Since $h$ is central in $H$ it does so injectively.  That is, the following diagram commutes
$$
\begin{diagram}
\node{\pi_1(F\times S^1)} \arrow{e,t}{\phi} \arrow{s} \node{H}\arrow{s,J}\\
\node{\pi_1(F\times \Sigma)}\arrow{e,t}{\overline\phi}\node{G}
\end{diagram}
$$
  and $\rho^{(2)}(F\times S^1,\phi)=\sigma^{(2)}(F\times \Sigma,\overline{\phi})-\sigma(F\times \Sigma)$.

We hope to use the behavior of residually finite $L^2$-signatures to compute $\sigma^{(2)}(F\times \Sigma,\overline\phi)$.  We must show that $G$ is residually finite.

\begin{lemma}\label{finite resolution}
Let $E=\langle a,b|\rangle$ be the rank two free group with generators $a$ and $b$.  Let $H$ be any residually finite group.  Let $h$ be an element of the center of $H$. 
  Then 
$G=\dfrac{H\oplus E}{h=[a,b]}$
 is residually finite.
\end{lemma}
\begin{proof}

We begin by noticing that in $G$, both $a$ and $b$ commute with $h=[a,b]$, so that $G\cong \dfrac{H\oplus E'}{h=[a,b]}$ where $E'=\langle a,b|[a,[a,b]],[b,[a,b]]\rangle$ is isomorphic to the 3-dimensional integral Heisenberg group.   That is, the multiplicative group of $3\times 3$ upper triangular integral matrices with ones on the main diagonal,
\begin{equation*}
E'\cong\left\{\left(\begin{array}{ccc}1&\alpha&\gamma\\0&1&\beta\\0&0&1\end{array}\right):\alpha,\beta,\gamma\in \Z\right\}\le Gl(3,\Z).
\end{equation*}

Let $e(\alpha,\beta,\gamma) = \left(\begin{array}{ccc}1&\alpha&\gamma\\0&1&\beta\\0&0&1\end{array}\right)$.  The isomorphism is given by  $a\mapsto e(1,0,0)$, $b\mapsto e(0,1,0)$.   For any integer $n$, let $\phi_n:E'\to E_n\le Gl(3,\Z_n)$ be the homomorphism given by reduction of entries mod $n$.  In $E_n$, $\phi_n([a,b])$ is order $n$.  

Recall that $G$ being residually finite is equivalent to the condition that for every $g\neq 1$ in $G$ there is a homomorphism to a finite group $f:G\to G'$ with $f(g)\neq 1$.

Let $g$ be a nontrivial element of $G$.  Then $g$ is represented by an equivalence class $(x\oplus y)$, $x\in H$, $y\in E'$.  If $y$ does not sit in the cyclic subgroup generated by $[a,b]$ in $E'$ (which is normal), then $(x\oplus y)$ is nonzero in the quotient, $q:G\to \frac{E'}{([a,b])}\cong \Z^2$, whose target is residually finite.  In some finite quotient of $\Z^2$, $q(g)=q(x\oplus y)$ is nonzero.

Alternately, if $y$ sits in the cyclic subgroup generated by $[a,b]$, then for some $m\in \Z$,  $g=(x\oplus [a,b]^m)=(xh^m\oplus 1)$ in $G$.  If $(x\oplus y)$ is nontrivial, then $x'=xh^m$ is nontrivial in $H$, which by assumption is residually finite.  Let $f:H\to \overline{H}$ be a homomorphism to a finite group with $f(x')\neq 1$.  Let $n$ be the order of $f(h)$ in this group.  The direct sum $f\oplus \phi_n:H\oplus E'\to \overline{H}\oplus E_n$ passes to a homomorphism $F:G\to \frac{\overline{H}\oplus E_n}{f(h)=\phi_n([a,b])}$.  

Consider the resulting commutative diagram
\begin{equation*}
\begin{diagram}
\node{H}\arrow{e}\arrow{s,l}{f}\node{H\oplus E'}\arrow{s,l}{f\oplus \phi_n}\arrow{e}\node{\frac{H\oplus E'}{h=[a,b]}}\arrow{s,l}{F}\\
\node{\overline{H}}\arrow{e}\node{\overline{H}\oplus E_n}\arrow{e}\node{\frac{\overline{H}\oplus E_n}{f(h)=\phi_n([a,b])}}
\end{diagram}
\end{equation*}

Since $h$ is central in $H$, $f(h)$ is central in $\overline{H}$.  Since $f(h)$ and $\phi_n([a,b])$ are each order $n$, the composition along the bottom row is injective.  By assumption, $f(x')\neq 1$,  We conclude that $F(g)=F(x'\oplus1)\neq 1$.  

Thus, any nontrivial element of $G$  is nontrivial in some finite quotient and $G$ is residually finite.  This completes the proof.

\end{proof}

So, $G$ is residually finite.  Let $G\ge G_1\ge G_2\ge \dots$ be a resolution for $G$ by finite index subgroups and $p_k:G\to G/G_k$ be the quotient map. Thus,
$$
\begin{array}{rll}
\rho^{(2)}(F\times S^1, \phi)&=\sigma^{(2)}(F\times \Sigma,\overline{\phi})-\sigma(F\times \Sigma)&\text{(Definition~\ref{defn of rho})}
\\
&=\underset{k\to \infty}{\lim} 
\sigma^{(2)}(F\times \Sigma,p_k\circ \overline{\phi})-\sigma(F\times \Sigma)
&\text{(Proposition~\ref{LuckSchickThm})}
\\
&= \underset{k\to \infty}{\lim} \rho^{(2)}(F\times S^1,p_k\circ \overline{\phi})&\text{(Definition~\ref{defn of rho}).}
\end{array}
$$
Since $p_k\circ \overline{\phi}:\pi_1(F\times S^1)\to G/G_k$ is a homomorphism to a finite group the second step of the proof applies to show $\rho^{(2)}(F\times S^1,p_k\circ \overline{\phi})=0$.  
This completes the proof.


\end{proof}

\section{Controlling the effect of local moves on the integral of the signature function.}


For the remainder of the paper, we study $\rho^0$, the $\rho$-invariant of $M(L)$ corresponding to abelianization.  To be precise, if $L$ is an $n$-component link with zero pairwise linking and $\phi:\pi_1(M(L))\to \Z^n$ is the abelianization homomorphism then $\rho^0(L):=\rho^{(2)}(M(L),\phi)$.  By Theorem~\ref{L2 Rho}, since $\widehat{\sigma_L}$ agrees with $\sigma_L$ away from a measure zero set,  $\rho^0(L)$ is given by the integral of the Cimasoni-Florens signature of $L$ over $\T^n_*$.  This section is devoted to controlling how  local moves can change this integral.

In \cite[Section 5]{CimFlo}, Cimasoni and Florens demonstrate the effect that crossing change and smoothing (see Figure~\ref{fig:smoothing}) have on the signature function.  We make use of these results in order to ease the approximation of the $\rho^0$-invariant.  Observe that for a link with non-zero linking numbers or a colored link with multiple components of a single color, we have demonstrated no relationship between the Cimasoni-Florens signature function and $\rho$-invariants.  Since we will need to discuss the integral of the Cimasoni-Florens signature of such links if we are to make use of the moves in Figure~\ref{fig:smoothing} we make up some notation.  For an arbitrary $n$-colored link $L$, the \textbf{$R$-invariant} of $L$, $R(L)$ is defined as the integral of ${\sigma_L}$ over the $n$-dimensional unit torus with respect to normalized  measure:
$$
R(L):=\frac{1}{(2\pi)^n}\int_{\T^n_*}{\sigma_L}(\omega)d\lambda(\omega)
$$

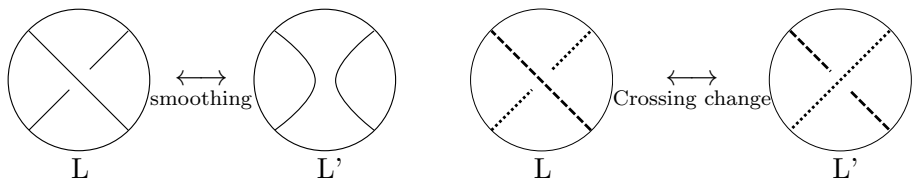
\begin{figure}[h!]
\setlength{\unitlength}{1pt}
\begin{picture}(340,60)
\put(0,10){\includegraphics[width=.15\textwidth]{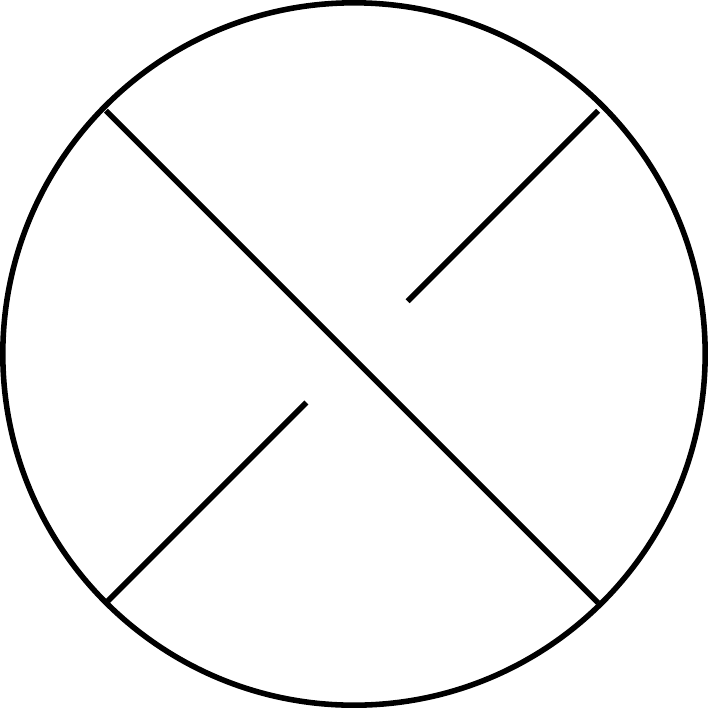}}
\put(24,0){L}
\put(54,35){$\underset{\text{smoothing}}{\longleftrightarrow}$}
\put(93,10){\includegraphics[width=.15\textwidth]{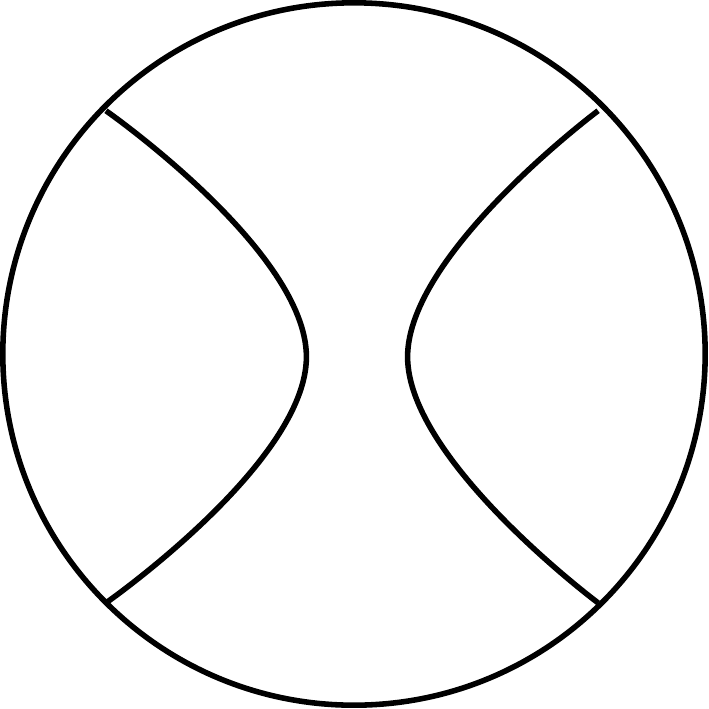}}
\put(117,0){L'}
\put(175,10){\includegraphics[width=.15\textwidth]{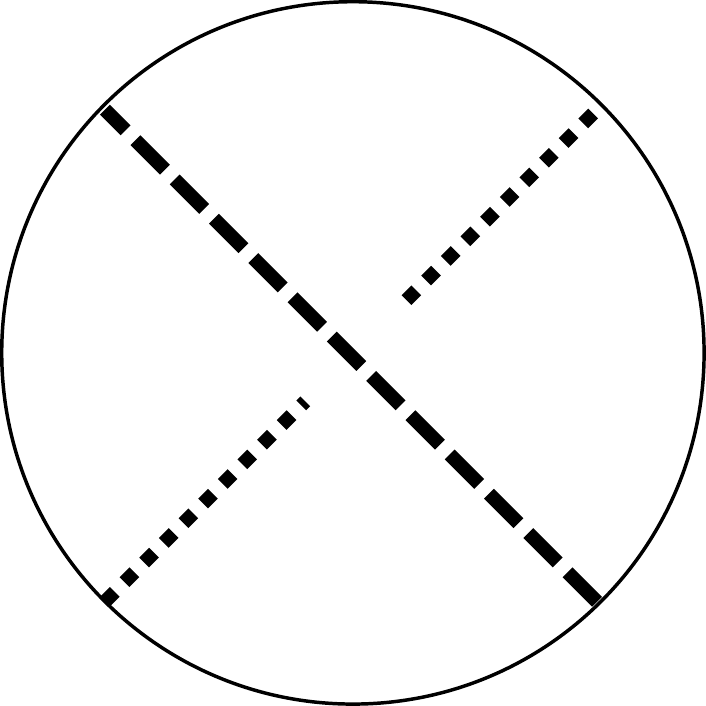}}
\put(199,0){L}
\put(229,35){$\underset{\text{Crossing change}}{\longleftrightarrow}$}
\put(288,10){\includegraphics[width=.15\textwidth]{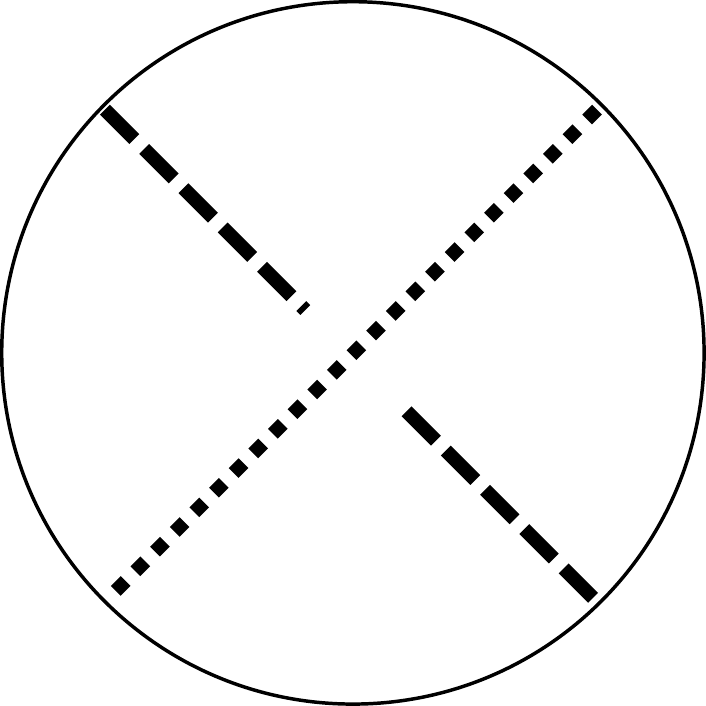}}
\put(312,0){L'}
\end{picture}
\caption{Moves which change the Cimasoni-Florens signature by $-1$, $0$, or $1$.  The move on the left is understood to preserve orientations and be between arcs of the same color.  The move of the right is understood to be between arcs of different colors.}\label{fig:smoothing}
\end{figure}

Cimasoni and Florens in \cite[Proposition 5.1]{CimFlo} show that if colored links $L$ and $L'$ differ by either of the the moves indicated in Figure~\ref{fig:smoothing} then for all $\omega$, $\sigma_L(\omega)$ and $\sigma_{L'}(\omega)$ differ by $0$ or $1$ or $-1$ depending on the Conway potential function of $L$ and $L'$.  Integrating this result gives that $R(L)$ and $R(L')$ differ by at most $1$.  

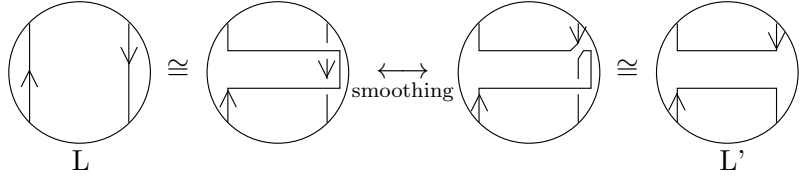
\begin{figure}[h!]
\setlength{\unitlength}{1pt}
\begin{picture}(310,60)
\put(0,10){\includegraphics[width=.15\textwidth]{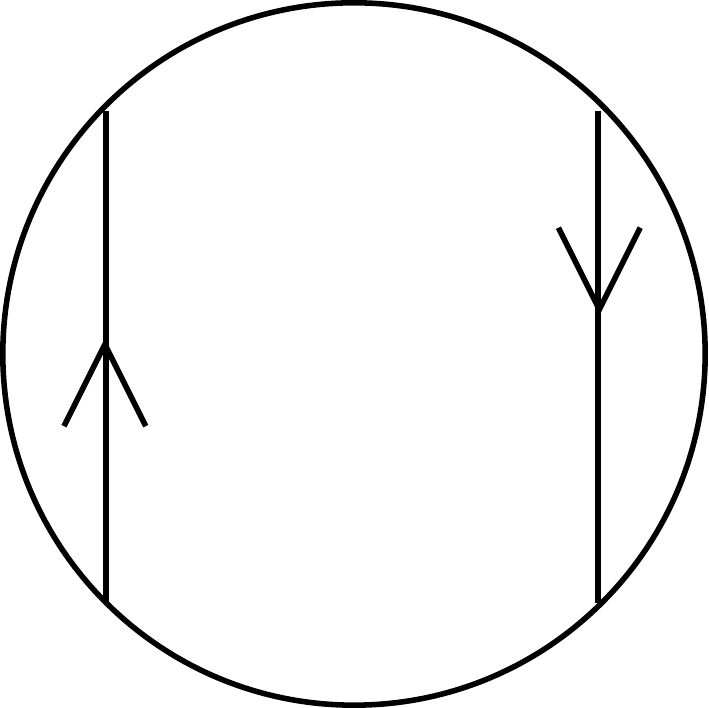}}
\put(24,0){L}
\put(60,35){$\cong$}
\put(75,10){\includegraphics[width=.15\textwidth]{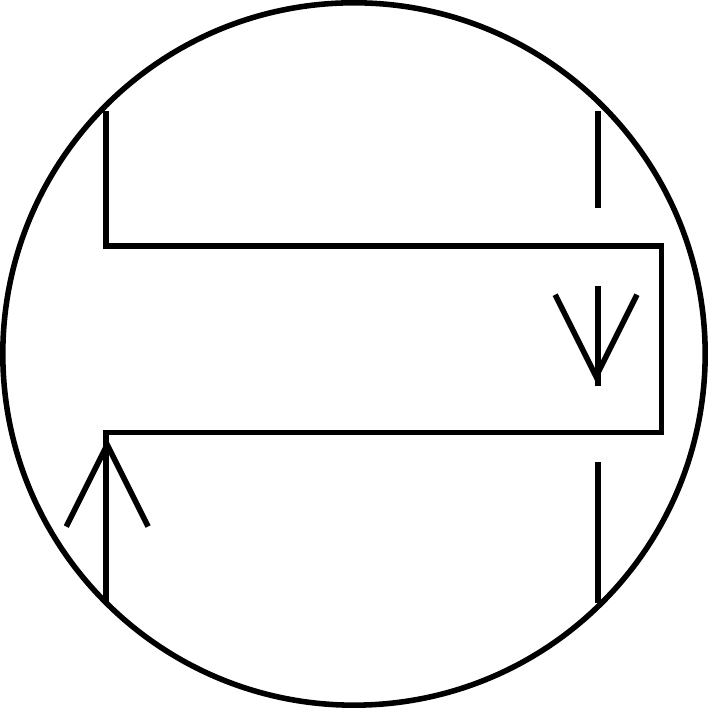}}
\put(130,35){$\underset{\text{smoothing}}{\longleftrightarrow}$}
\put(170,10){\includegraphics[width=.15\textwidth]{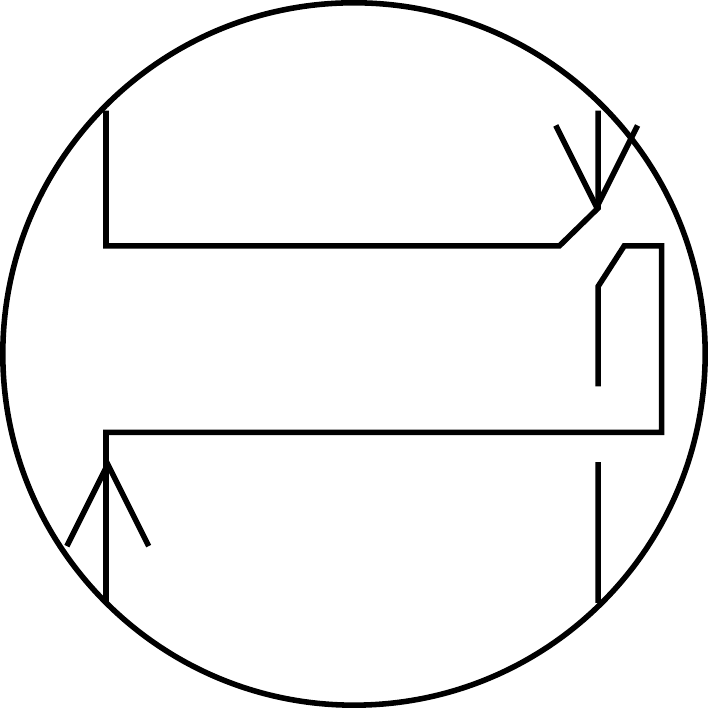}}
\put(230,35){${\cong}$}
\put(245,10){\includegraphics[width=.15\textwidth]{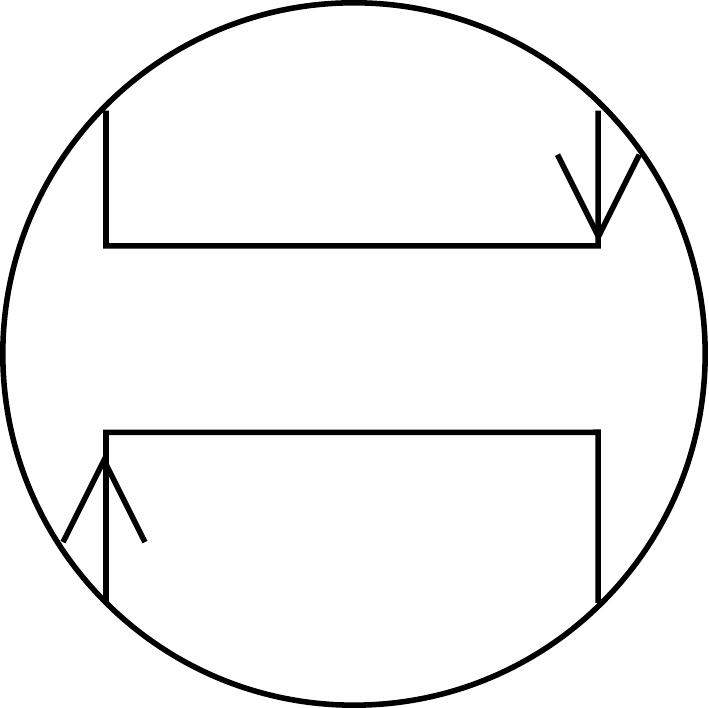}}
\put(269,0){L'}
\end{picture}
\caption{A band addition can be realized via a smoothing.}\label{fig:band move}
\end{figure}

The addition of a band between two arcs of the same color (as in Figure~\ref{fig:band move}) can be realized by a single smoothing.  Thus, if $L$ and $L'$ differ by either a band addition or a crossing change between arcs of the same color, then $R(L)$ and $R(L')$ differ by at most $1$.  Additionally, \cite[Proposition 2.12]{CimFlo} shows that if this band runs between split sublinks of $L$, then $\sigma_L=\sigma_{L'}$.  Summarizing these results,

\begin{proposition}\label{local moves}
If $L$ and $L'$ differ by any of the local moves of Figures~\ref{fig:smoothing} of \ref{fig:band move} then $|R(L)-R(L')|\le 1$.  If the band in Figure~\ref{fig:band move} runs between split sublinks, then $R(L)=R(L')$
\end{proposition}

The local move on colored links depicted in Figure~\ref{fig:band twist} will be relevant in the next section.  We provide bounds on how it changes the $R$-invariant.

\begin{figure}[h!]
\setlength{\unitlength}{1pt}
\begin{picture}(230,55)
\put(0,0){\includegraphics[height=.15\textwidth]{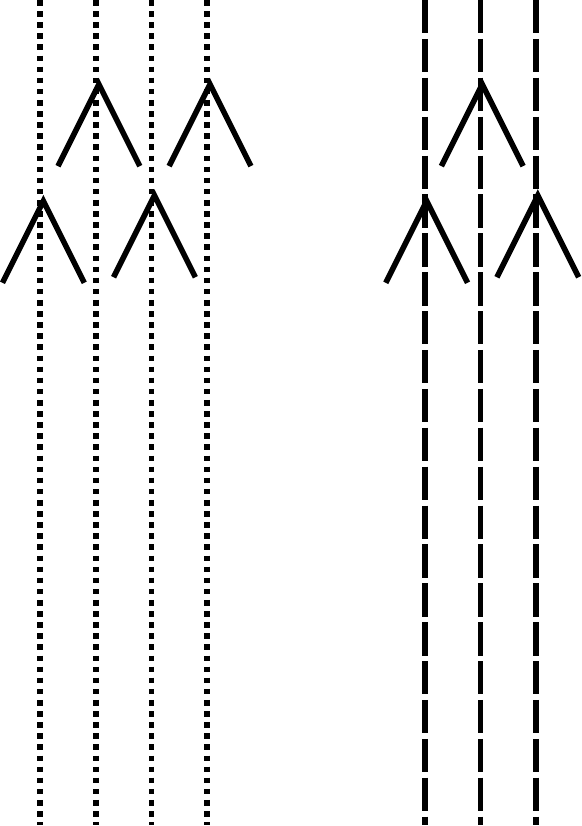}}
\put(43,25){$\leftrightarrow$}
\put(60,0){\includegraphics[height=.15\textwidth]{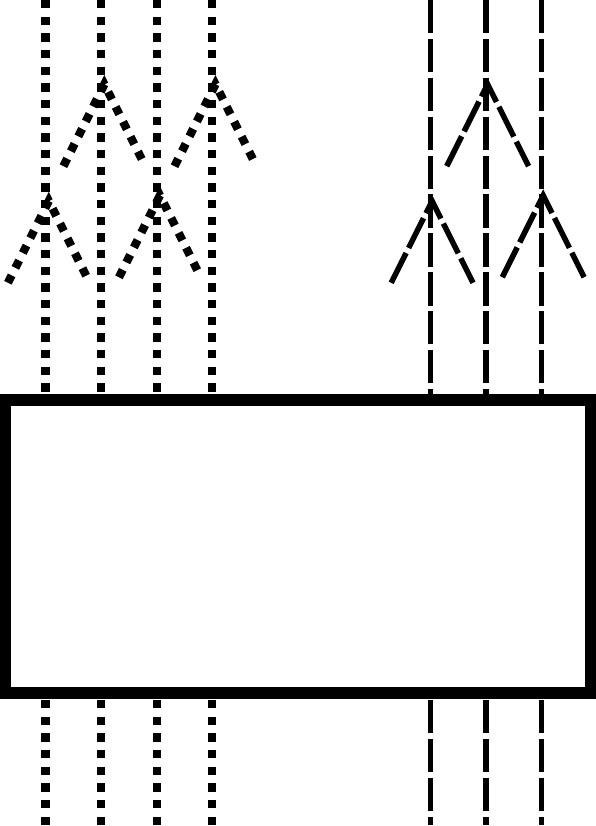}}
\put(75,15){$B$}
\put(103,25){$:=$}
\put(120,0){\includegraphics[height=.15\textwidth]{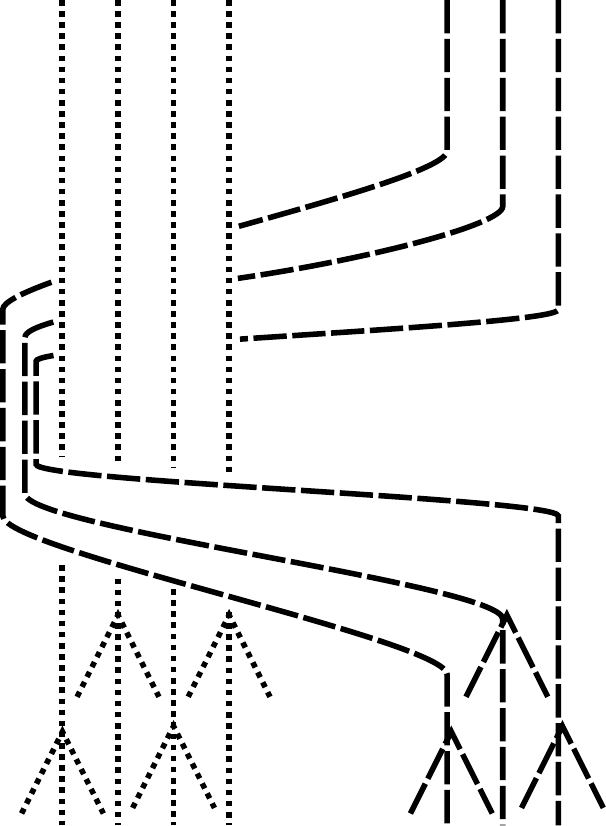}}
\put(163,25){$\cong$}
\put(180,0){\includegraphics[height=.15\textwidth]{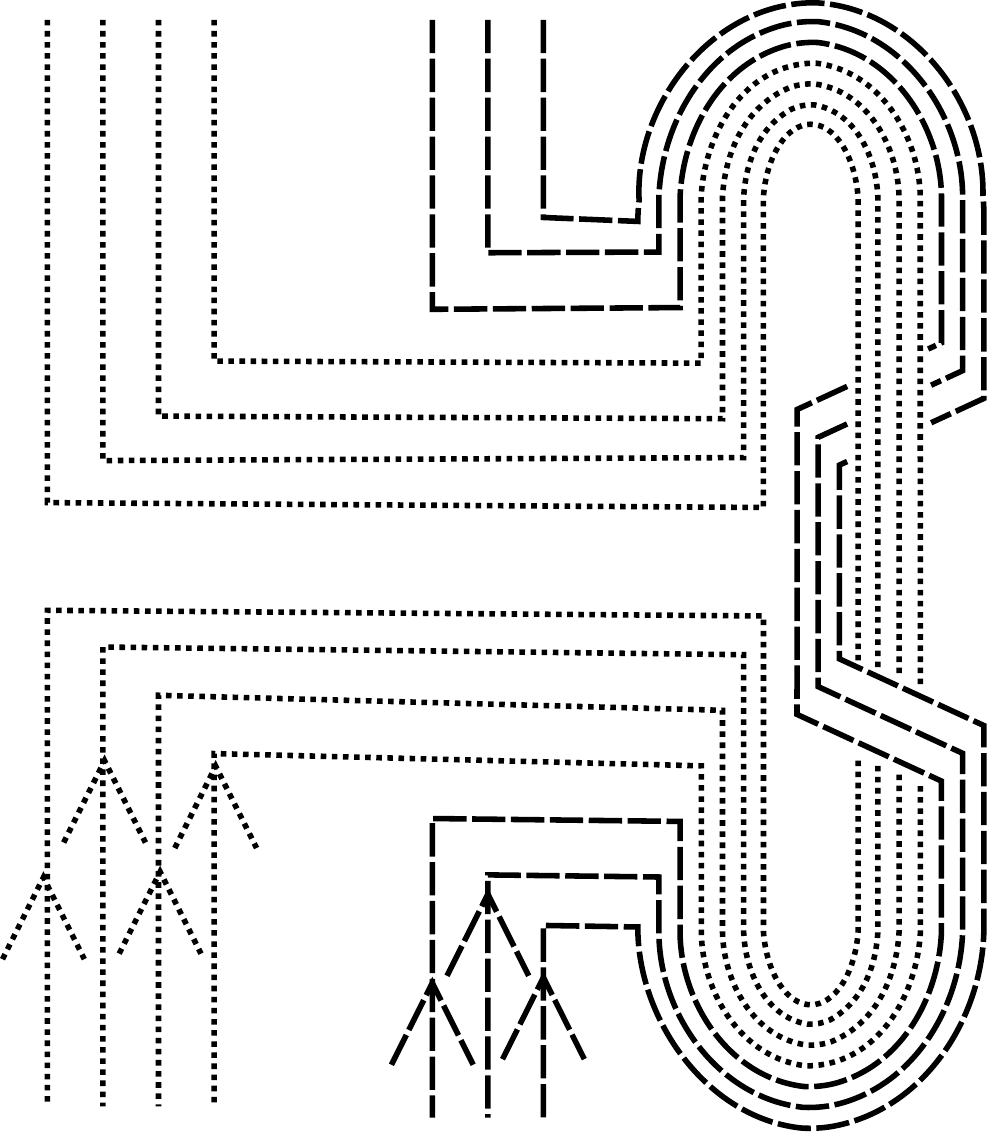}}
\end{picture}
\caption{The move considered in Proposition~\ref{VMoveProp} realized via  band summing with the link $V$.  The two different bands of arcs are assumed to be of different colors.}\label{fig:band twist}
\end{figure}
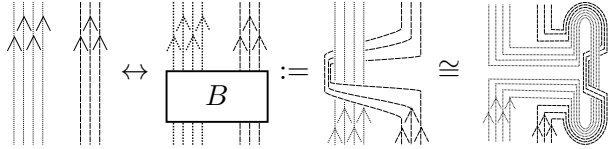

\begin{proposition}\label{VMoveProp}
If the colored links $L$ and $L'$ differ by the move depicted in Figure~\ref{fig:band twist} with $a$ strands of one color and $b$ strands of a different color then $|R(L)-R(L')|\le a+b-1$
\end{proposition}
\begin{proof}
As is shown in Figure~\ref{fig:band twist}, $L'$ is given by taking the split union of $L$ with the two color link $V$ depicted in Figure~\ref{fig:LinkV} and adding $a+b$ bands, with the first band being between split sublinks. Since the signature invariant adds under split union, by \cite[Proposition 2.12]{CimFlo}, it follows that $|R(L')-R(L)-R(V)|\le a+b-1$, where $V$ is the link depicted in Figure~\ref{fig:LinkV}.

\begin{figure}[h!]
\setlength{\unitlength}{1pt}
\begin{picture}(110,110)
\put(0,0){\includegraphics[height=100pt]{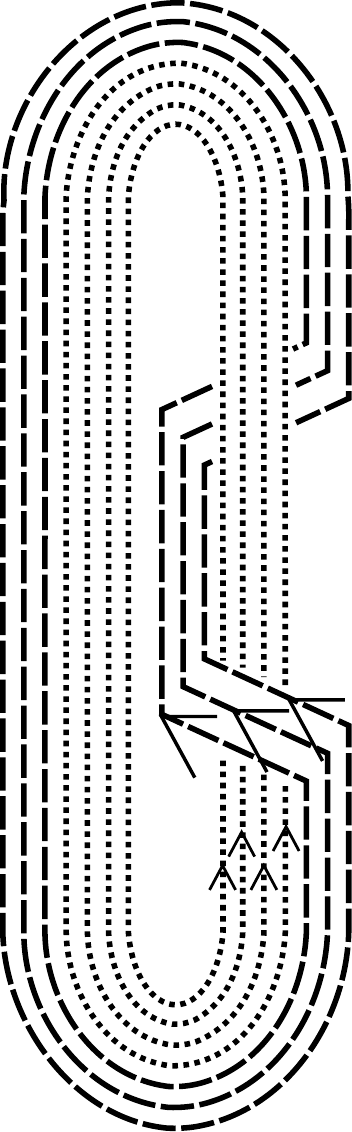}}
\put(60,40){$\cong$}
\put(90,0){\includegraphics[height=100pt]{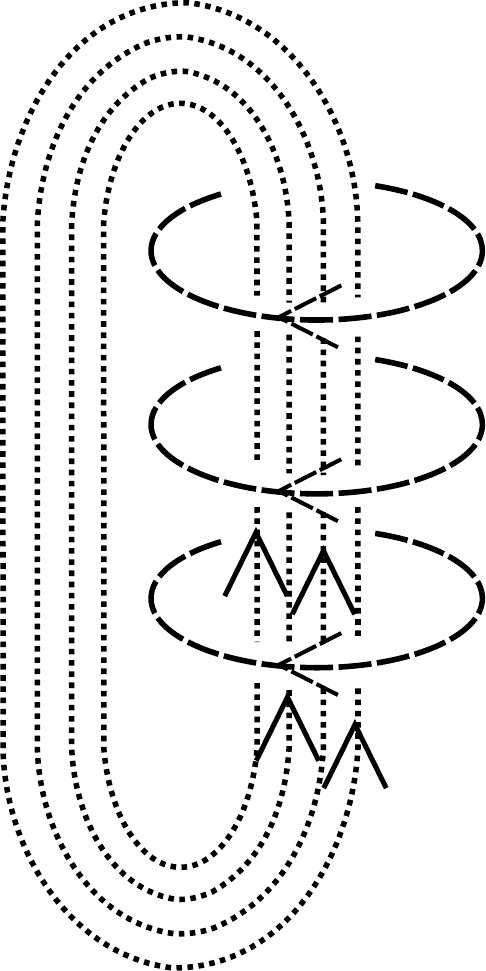}}
\end{picture}
\caption{The link $V$ and an isotopy.}\label{fig:LinkV}
\end{figure}

In order to compute $R(V)$, we notice that the result of reversing the orientation all of the components of $V$ of one color is isotopic the mirror image of $V$, so that by \cite[Proposition 2.10 and Corollary 2.11]{CimFlo}, $\sigma_V(\omega_1,\omega_2^{-1})=-\sigma_V(\omega_1,\omega_2)$.  Since $\omega\mapsto\omega^{-1}$ is a measure preserving transformation of $\T^1$, this implies that $R(V)=0$ and completes the proof.
\end{proof}

\section{Application: The linear independence of the twist knots which are algebraically of order 2}\label{application}

In this section, we study those twist knots (depicted in Figure~\ref{fig:twist}) which are of order 2 in the algebraic concordance group.

We recall some concepts from algebraic concordance. Let $K$ be an algebraically slice knot with a genus $g$ Seifert surface $F$.  By the definition of algebraic concordance, there exists a nonseparating collection of $g$ simple closed curves on $F$,  $L=L_1\dots L_g$, for which the Seifert form of $F$ vanishes.  That is, $\lnk(L_i, L_j^+)=0$ for all $i,j$, where $L_j^+$ is the result of pushing $L_j$ off of $F$ in the positive normal direction and $\lnk$ is the the linking number.  Using language from \cite{derivatives}, $L$ is called a \textbf{derivative} of $K$.  

For a simple closed curve, $\gamma$, on $F$, a curve $m$ in the complement of $F$ is called a \textbf{meridian for the band on which $\gamma$ sits} if $m$ bounds a disk which intersects $F$ transversely in a single arc and intersects $\gamma$ transversely in a single point.

In \cite{MyFirstPaper} the author defines a  $\rho$-invariant of knots, $\rho^1$ and proves the following results.  

\begin{proposition*}[See \cite{MyFirstPaper}, Corollary 4.4]\label{MFPobs}
The set of all twist knots $T_n$ for which $\rho^1(T_n)\neq 0$ and  $T_n$ is of order 2 or 4 in $\AC$  is linearly independent in the knot concordance group.
\end{proposition*}

\begin{proposition*}
[\cite{MyFirstPaper}, Theorem 5.6]\label{MFPbound}
Suppose that $K$ is a genus 1 knot and that $K\#K$ is algebraically slice.  Let $L$ be a two component derivative of $K\#K$.  If the components of $L$ together with the meridians of the bands on which $L$ sit form a  $\Z$-linearly independent set in the Alexander module of $K\#K$, then 
$|2\rho^1(K)-\rho^0(L)|\le 1.$
\end{proposition*}

Let $T_n$ be a twist knot for which $T_n\# T_n$ is order 2 in $\AC$.  According to \cite[Corollary 4.4]{MyFirstPaper} and Theorem~\ref{L2 Rho}, in order to show that  $\rho^1(T_n)\neq 0$, it suffices to find a derivative, $L$, of $T_n\#T_n$ for which the integral of  $\sigma_{L}$ over $\T^2$ is greater than $1$ in absolute value.  

 For every $n$ such that $T_n\#T_n$ is algebraically slice we find a derivative,  
namely the link of Figure~\ref{fig:derivative}.

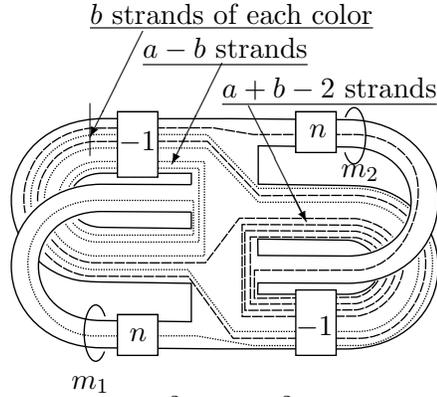
\begin{figure}[h!]
\setlength{\unitlength}{1pt}
\begin{picture}(200,150)
\put(0,0){\includegraphics[width=.45\textwidth]{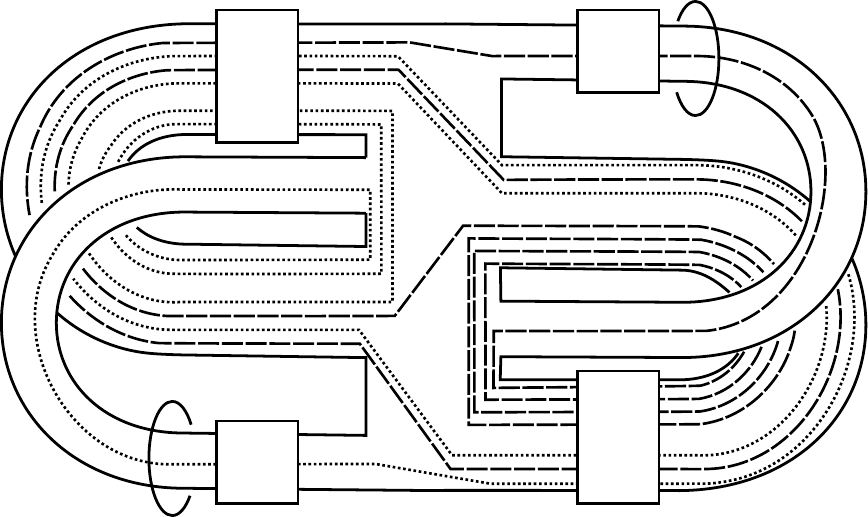}}
\put(45,8){$n$}
\put(108,10){$-1$}
\put(41,80){$-1$}
\put(113,85){$n$}
\put(125,70){$m_2$}
\put(23,-10){$m_1$}
\put(30,128){\underline{$b$ strands of each color}}
\put(30,78){\line(0,1){20}}
\put(50,125){\vector(-1,-2){20}}
\put(50,115){\underline{$a-b$ strands}}
\put(80,113){\vector(-1,-2){19}}
\put(80,100){\underline{$a+b-2$ strands}}
\put(90,97){\vector(1,-2){22}}
\end{picture}
\caption{For $n=a^2-a+b^2$, the link $L_{a,b}$ is a derivative for $T_n$  The curves $m_1$ and $m_2$ are meridians about the bands on which the components of $L$ sit.}\label{fig:derivative}
\end{figure}

\begin{lemma}\label{order2}
The $n$-twist knot $T_n$ is algebraically of order exactly 2 if and only if there are positive integers $a$ and $b$ such that $n=a^2-a+b^2$ but there does not exist an integer $c$ with $n=c^2-c$.  For such an $n$, the link $L_{a,b}$, depicted in Figure~\ref{fig:derivative}, is a derivative for $T_n\#T_n$.  Moreover the meridians, $m_1$ and $m_2$, of the bands on which the components of $L_{a,b}$ sit are $\Z$-linearly independent of the components of $L_{a,b}$ in the Alexander module of $T_n\#T_n$.
\end{lemma}
\begin{proof}

According to Levine \cite[Corollary 23]{Le10}, $T_n\#T_n$ is algebracially slice if and only if $n>0$ and $4n+1$ has no odd multiplicity prime factors congruent to $3$ mod $4$.  Under these conditions an elementary fact from number theory (see \cite[Theorem 12.3]{Bu}, for example) provides integers $a_0$ and $b_0$ with $a_0^2+b_0^2=4n+1$.  Since $4n+1$ is odd it must be that $a_0=2a-1$ is odd while $b_0=2b$ is even.  Thus, $a^2-a+b^2=n$.  If $n=c^2-c$ for some $c$ then $T_n$ is algebraically slice.

\begin{figure}[h!]
\setlength{\unitlength}{1pt}
\begin{picture}(200,90)
\put(0,0){\includegraphics[width=.45\textwidth]{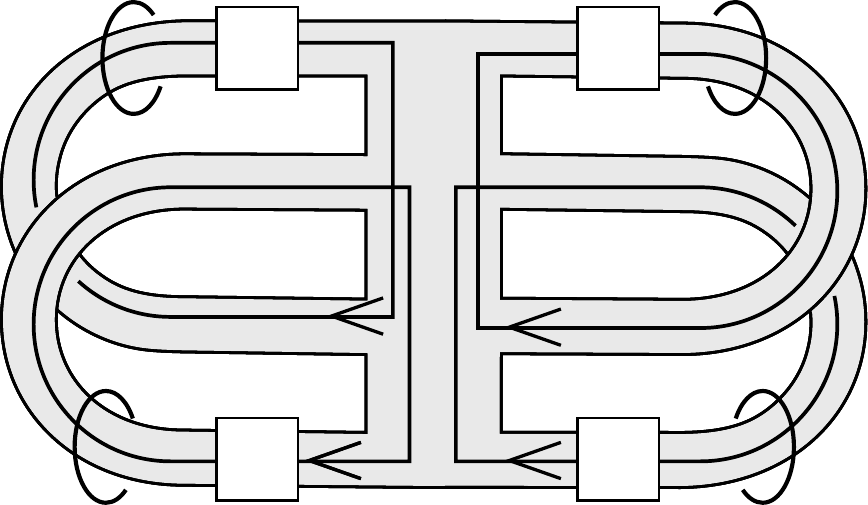}}
\put(45,5){$n$}
\put(108,5){$-1$}
\put(41,83){$-1$}
\put(113,82){$n$}
\put(-4,15){$\gamma_1$}
\put(-4,80){$\gamma_2$}
\put(155,10){$\gamma_4$}
\put(155,80){$\gamma_3$}
\put(8,-5){$\mu_1$}
\put(8,90){$\mu_2$}
\put(145,-5){$\mu_3$}
\put(145,90){$\mu_4$}
\end{picture}
\caption{$\gamma_1,\dots,\gamma_4$ form a basis for the first homology of a Seifert surface for $T_n\#T_n$.  The meridians for $\gamma_i$, $\mu_1,\dots, \mu_4$ form a generating set for the Alexander module of $T_n\#T_n$.}\label{fig:surfaceWithBasis}
\end{figure}
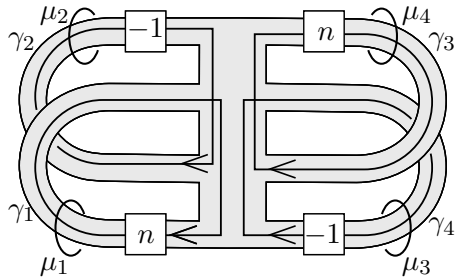

We now prove that the link $L_{a,b}$ is a derivative.   The first homology of the Seifert surface $F$ for $T_n\# T_n$ is a free abelian group with basis given by the curves $\gamma_1,\gamma_2,\gamma_3,\gamma_4$ in Figure~\ref{fig:surfaceWithBasis}.  With respect to this basis the components of $L_{a,b}$ represent the classes in $H_1(F)$ given by the vectors
$$
v_1=\left[\begin{array}{cccc}1&a&0&b\end{array}\right]^T, 
v_2=\left[\begin{array}{cccc}0&b&1&(1-a)\end{array}\right]^T
$$
while the Seifert form of $F$ is given the matrix
$$
V=\left[\begin{array}{cccc}n&1&0&0\\0&-1&0&0\\0&0&n&1\\0&0&0&-1\end{array}\right]
$$
A computation (remembering that $n=a^2-a+b^2$) verifies that $v_i^TVv_j=0$  for $i,j\in \{1,2\}$.  Thus, $L_{a,b}$ is a derivative.

In order to address the linear independence claim we first recall some classical knot theory facts.  A good reference is \cite[Chapter 8, Section C]{Rolfsen}.  The (rational) Alexander module, $A_0(T_n\#T_n)$, is the $\Q[t,t^{-1}]$-module generated by $\mu_1,\mu_2,\mu_3,\mu_4$ with presentation matrix $V-tV^T$.  With respect to this presentation, $m_1=\mu_1$ and $m_2=\mu_3$ correspond to the first and third generators and $\gamma_i$ corresponds to the $i^\text{th}$ column of $V$.  

Performing the presentation calculus needed to unwind this yeilds  
$$A_0(T_n\#T_n)\cong \left(\frac{\Q[t,t^{-1}]}{\left(nt^2-(2n+1)t+n\right)}\right)\oplus\left(\frac{\Q[t,t^{-1}]}{\left(nt^2-(2n+1)t+n\right)}\right)$$
with generators $m_1$ and $m_2$ and that the components of $L_{a,b}$ (represented by $v_1$ and $v_2$) correspond to 
\begin{eqnarray*}v_1&\mapsto&(n+a(1-n+nt))m_1+b(1-n+nt)m_2\\
v_2&\mapsto &b(1-n+nt)m_1+(1+(1-a)(1-n+nt))m_2.\end{eqnarray*}

Provided $n\neq 0$ it is straightforward to see that images of $m_1,m_2,v_1$ and $v_2$ form a $\Q$-linearly independent set.  This completes the proof.
\end{proof}

According to \cite[Corollary 4.4]{MyFirstPaper} and Theorem~\ref{L2 Rho}, in order to show that $\rho^1(T_n)\neq 0$, it suffices to show that the integral of the Cimasoni-Florens signature of the link $L_{a,b}$ over $\T^2$ is greater than $1$ in absolute value.  The remainder of this paper is devoted to this computation.



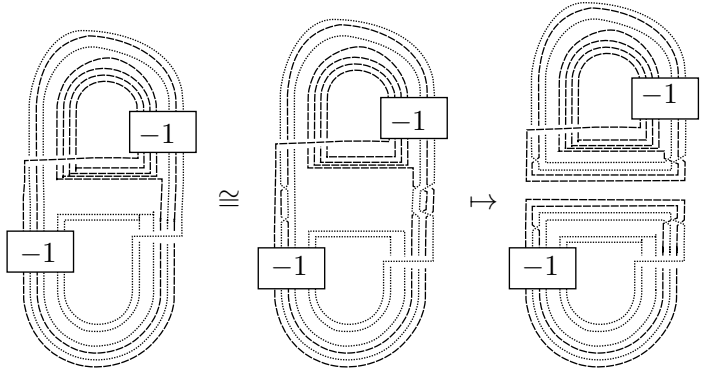
\begin{figure}[h!]
\setlength{\unitlength}{1pt}
\begin{picture}(300,140)
\put(0,0){\includegraphics[height=.2\textwidth, angle=90]{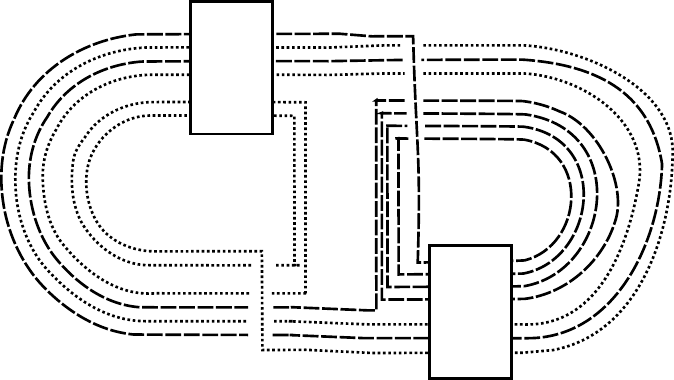}}
\put(5,40){$-1$}
\put(50,85){$-1$}
\put(80,60){$\cong$}

\put(95,0){\includegraphics[height=.2\textwidth, angle=90]{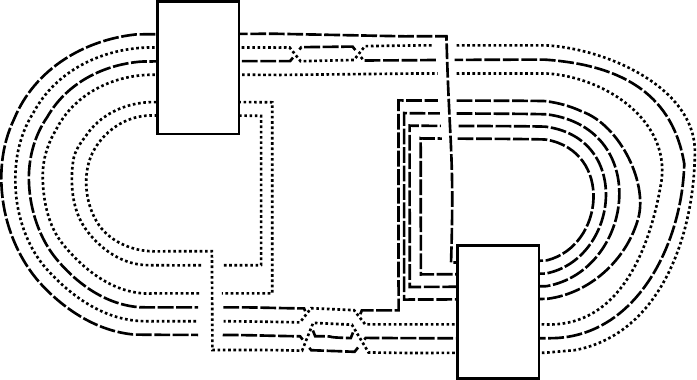}}
\put(100,35){$-1$}
\put(145,90){$-1$}
\put(175,60){$\mapsto$}

\put(190,0){\includegraphics[height=.2\textwidth, angle=90]{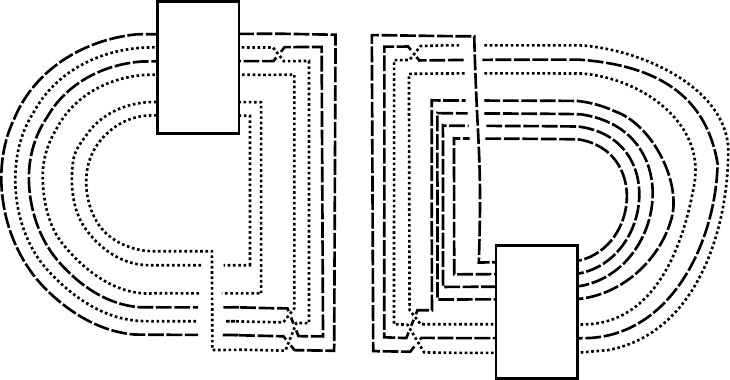}}
\put(195,35){$-1$}
\put(240,100){$-1$}
\end{picture}
\caption{Left to right: The link $L_{a,b}$, an isotopy and $L_{a,b}^1$, the result of adding $b$ bands to $L_{a,b}$.}\label{fig:Lab0}
\end{figure}

As its depicted in Figure~\ref{fig:Lab0}, by adding $2b$ bands to this link, $b$ of each color, the link $L_{a,b}$ is reduced to the split union of two links, $L^1_{a,b}$.  By Proposition~\ref{local moves}, $|\rho^0(L_{a,b})-R(L^1_{a,b})|\le 2b-1$.

\begin{figure}[h!]
\setlength{\unitlength}{1pt}
\begin{picture}(300,130)
\put(0,0){\includegraphics[width=.39\textwidth, angle=90]{DerivativeAddBands.pdf}}
\put(5,35){$-1$}
\put(55,98){$-1$}
\put(85,60){$\cong$}

\put(100,0){\includegraphics[width=.39\textwidth, angle=90]{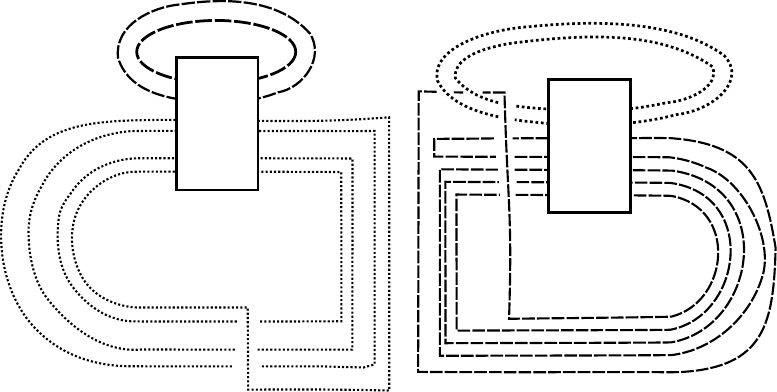}}
\put(115,35){$-1$}
\put(118,105){$-1$}

\put(185,60){$\mapsto$}
\put(200,0){\includegraphics[width=.39\textwidth, angle=90]{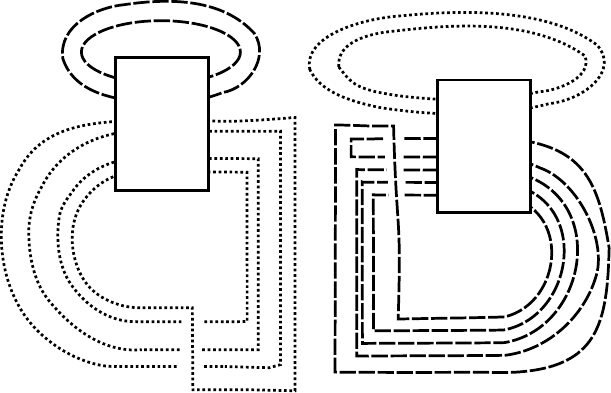}}
\put(222,34){$-1$}
\put(226,110){$-1$}
\end{picture}
\caption{Left to right: The link $L^1_{a,b}$, an isotopy of it and $L_{a,b}^2$ the result of changing $b$ crossings.}\label{fig:Lab1}
\end{figure}

As is depicted in Figure~\ref{fig:Lab1}, by changing $b$ crossings between components of different colors $L_{a,b}^1$ becomes the link $L_{a,b}^2$, so that $|R(L_{a,b}^1)-R(L_{a,b}^2)|\le b$ and $|\rho^0(L_{a,b})-R(L^2_{a,b})|\le 3b-1$.

\begin{figure}[h!]
\setlength{\unitlength}{.9pt}
\begin{picture}(370,140)
\put(0,0){\includegraphics[width=.39\textwidth, angle=90]{AddBandsLoseCrossings.pdf}}
\put(22,34){$-1$}
\put(26,120){$-1$}
\put(114,75){$\cong$}
\put(120,0){\includegraphics[width=.39\textwidth, angle=90]{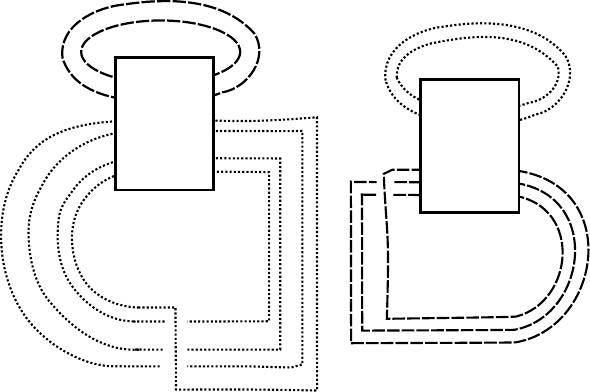}}
\put(142,40){$-1$}
\put(146,120){$-1$}
\put(235,75){$\cong$}
\put(245,0){\includegraphics[width=.39\textwidth,angle=90]{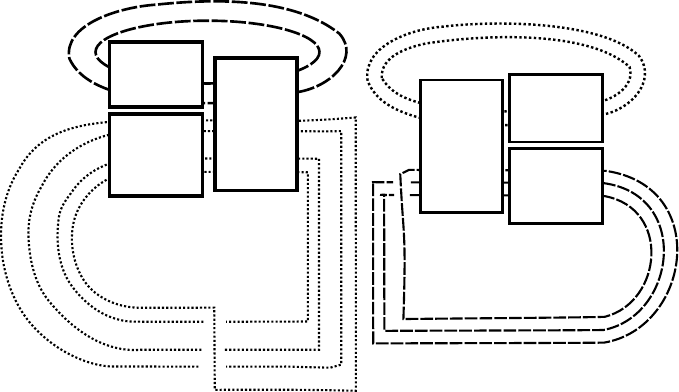}}
\put(266,55){$B$}
\put(273,30){$-1$}
\put(254,30){$-1$}
\put(274,100){$B$}
\put(262,125){$-1$}
\put(280,125){$-1$}
\end{picture}
\caption{Left to right: The link $L^2_{a,b}$ and some isotopies of it.  The $B$ denotes the move in Figure~\ref{fig:band twist}.}\label{fig:Lab2}
\end{figure}

The link $L_{a,b}^2$ is the result of performing two band twist moves (as in Figure~\ref{fig:band twist}) starting with the link $L_{a,b}^3$ of Figure~\ref{fig:Lab3}.  The latter is the split union of four torus links.  The leftmost twist involves $a$ strands of one color and $b$ of the other.  The rightmost involves $a-1$ of one color and $b$ of the other.  By Proposition~\ref{VMoveProp} $|R(L^2_{a,b})-R(L_{a,b}^3)|\le(a+b-1)+(a-1+b-1)$, and $|\rho^0(L_{a,b})-R(L_{a,b}^3)|\le 5b+2a-4$.

Finally,  by \cite[Proposition 2.13]{CimFlo} Cimasoni-Florens signature adds under split union:
 $$
 R(L_{a,b}^3)=R(T(a,1-a))+R(T(1-a,a))+2R(T(b,-b)),
 $$
  where $T(a,b)$ is the one color $(a,b)$-torus link.  In the case of knots, the Cimasoni-Florens signature agrees with the Tristram-Levine signature.   The integral of the Tristram-Levine signature of torus knots is computed by Borodzic \cite{Boro1} and independently by Collins \cite{Collins}.  In \cite{Boro1}, Borodzic also computes the integral of the one colored signature of the $(b,b)$-torus link, which is the mirror image of the $(b,-b)$-torus link.  Specifying their results to our setting,
\begin{equation*}
\begin{array}{rclll}
R(T(a,1-a))&=&R(T(1-a,a))&=&\dfrac{(a+1)(a-2)}{3}\\
R(T(b,-b))&=&\dfrac{(b-1)^2}{3}.
\end{array}
\end{equation*}

\begin{figure}[h!]
\setlength{\unitlength}{1pt}
\begin{picture}(140,120)
\put(0,0){\includegraphics[height=.25\textwidth, angle=90]{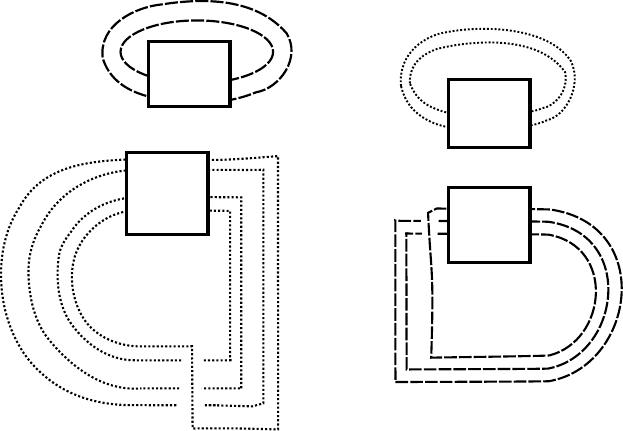}}
\put(8,36){$-1$}
\put(34,30){$-1$}
\put(17,100){$-1$}
\put(40,100){$-1$}
\end{picture}
\caption{The link $L^3_{a,b}$ consists of the split union of two one color $(b,-b)$-torus links, a $(a-b,b-a+1)$-torus knot and a $(a+b-3,-a-b+2)$-torus knot.}\label{fig:Lab3}
\end{figure}
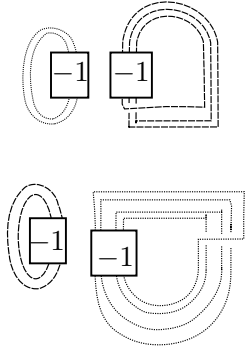

Putting all of this together,
\begin{equation}
\begin{array}{rcl}
\rho^0(L_{a,b})&\ge&\dfrac{2(a+1)(a-2)+2(b-1)^2}{3}-(2a+5b-4)\\
&=&\dfrac{2a^2+2b^2-8a-19b+10}{3}.
\end{array}
\end{equation}


Since we are interested in when $\rho^0(L_{a,b})>1$, we ask when
$$f(a,b):=2a^2+2b^2-8a-19b+7>0.$$

By observing that $f(a,b)$ grows quadratically in both $a$ and $b$, we immediately see that
\begin{proposition*}
For all but finitely many $a$ and $b$, $f(a,b)>0$ so that $\rho^0(L_{a,b})>1$ and for $n=a^2-a+b^2$, $\rho^1(T_{n})>0$.  Thus, the set containing all of the twist knots $T_n$ which are algebraically of order $2$ is linearly independent with only finitely many exceptions.
\end{proposition*}

Next, We use a computer to determine $\{n:n=a^2-a+b, a>0, b>0, f(a,b)>0\}$.  Combining this result with the computation in \cite[Corollary 6.2]{MyFirstPaper} showing that $\rho^1(T_n)>0$ for $n$ of the form $n=x^2+x+1$ with $x>1$ gives the following result.   (The knot called $T_n$ in this paper is the mirror image of what is called $T_{-n}$ in \cite{MyFirstPaper}) 

\begin{proposition}\label{linear indep second try}
The set containing all of the twist knots $T_n$ which are algebraically of order $2$ is linearly independent with the following 39 possible exceptions:
\begin{equation*}
\begin{array}{rcl}
n&=&1, 3, 4, 
9, 10, 11, 
15, 16, 18, 
22, 24, 25, 27, 28, 29, 
34, 36, 37, 38, 
\\&&
39, 
45, 48, 49, 51, 55, 
58, 61, 64, 66, 67, 69, 70, 78, 79, 83, 84, 87,
\\&&
 93, 101
.
\end{array}
\end{equation*}\end{proposition}

One can do better by noticing the following consequence of the results of \cite{MyFirstPaper}.  

\begin{proposition} Let $A$ be a set of knots with distinct prime Alexander polynomials, vanishing $\rho^0$-invariants and nonvanishing $\rho^1$-invariants.  Let $B$ be a linearly independent set of knots with vanishing $\rho^0$-invariants and prime Alexander polynomials distinct from the Alexander polynomials of the elements of $A$, then $A\cup B$ is linearly independent.  
\end{proposition}
\begin{proof}
The proof relies on a family von Neumann $\rho$-invariants indexed by the polynomial $p(t)$ defined in \cite{MyFirstPaper}, $\rho^1_{p(t)}$.  The following proposition reduces them in many cases to $\rho^0$ or $\rho^1$.

\begin{proposition*}[\cite{MyFirstPaper} Proposition 3.4]
Let $\Delta(t)$ be the Alexander polynomial of a knot $K$.  Then
\begin{enumerate}
\item  $\rho^1_{\Delta(t)}(K)=\rho^1(K)$.
\item If $(p,\Delta)=1$, then $\rho^1_{p(t)}(K)=\rho^0(K)$.
\end{enumerate}
\end{proposition*}

Moreover, by \cite[Theorem 4.1 and Proposition 3.5]{MyFirstPaper}, $\rho^1_p$ is a homomorphism when restricted to the subgroup of $\conc$ generated by knots with prime Alexander polynomials.

Suppose that some linear combination $\left(\underset{K\in A}{\#}{l_K K}\right)\#\left(\underset{J\in B}{\#}{m_J J}\right)$ is slice.  Let $p$ be the Alexander polynomial of some $K_0\in A$.  Since $\rho^1_p$ is a homomophism, it follows that 
\begin{equation}\label{lineardep}0=\sum_{K\in A}l_K\rho^1_p(K)+\sum_{J\in B} m_J\rho^1_p(J).\end{equation}
Next, by \cite[Proposition 3.4]{MyFirstPaper}, and since $\rho^0$ vanishes for all of the knots in $A$ and $B$, equation~\eref{lineardep} reduces to $0=l_{K_0}\rho^1(K_0)$.  Since $\rho^1(K_0)\neq0$, it follows that $l_{K_0}=0$.  

Since $K_0$ was arbitrarily chosen in $A$, it follows that every  $l_K$ vanishes.  Finally, since $B$ is assumed to be linearly independent in $\conc$, this implies that every $m_J=0$.  We conclude that $A\cup B$ is linearly independent. 
\end{proof}

In \cite[Corollary 1.2]{Tamulis}, Tamulis finds that $\{T_n:4n+1\text{ is prime}\}$ is linearly independent in $\conc$.  Hence, we can remove all elements which satisfy this condition from the set of possible counter-examples of Proposition~\ref{linear indep second try}.

\begin{theorem}\label{linear indep done}
The set containing all of the twist knots $T_n$ which are algebraically of order $2$ is linearly independent with the following $12$ possible exceptions:
$$
\begin{array}{rcl}
n&=&1, 
11, 
16, 
29, 
36, 
38, 
51, 55, 
61, 
66, 
83, 
101
.
\end{array}
$$
\end{theorem}


\bibliographystyle{plain}
\bibliography{biblio}  

\end{document}